\documentclass{amsart} 
\usepackage{nicefrac}
\usepackage{amssymb, amsmath,mathtools}
\usepackage{mathrsfs}
\usepackage{amscd}
\usepackage{verbatim}
\usepackage{stmaryrd}
\usepackage[dvipsnames]{xcolor}
\usepackage{a4wide} 
\usepackage[sort,numbers,sort&compress]{natbib}

\usepackage{enumitem}

\usepackage[colorlinks,linkcolor={blue},citecolor={blue},urlcolor={purple}]{hyperref}

\usepackage[colorinlistoftodos,prependcaption,textsize=tiny]{todonotes}

\usepackage{forest}
\usepackage{tikz}
\tikzset{>=latex}

\usepackage{tabularx}
\newcolumntype{L}{>{\arraybackslash}X}
\usepackage{multirow}

\theoremstyle{plain}
\newtheorem{theorem}{Theorem}[section]
\theoremstyle{remark}
\newtheorem{remark}[theorem]{Remark}

\theoremstyle{plain}
\newtheorem{corollary}[theorem]{Corollary}
\newtheorem{lemma}[theorem]{Lemma}
\newtheorem{proposition}[theorem]{Proposition}
\newtheorem{definition}[theorem]{Definition}

\newtheorem{assumption}[theorem]{Assumption}

\numberwithin{equation}{section}

\newcommand{\N}{\mathbb{N}}
\newcommand{\Z}{\mathbb{Z}}

\newcommand{\R}{\mathbb{R}}

\newcommand{\Cper}{C(D)}
\newcommand{\ceqq}{\coloneqq}

\newcommand{\col}{\colon}  
\newcommand{\pf}{p_{\kern-.5pt \alpha}}

\newcommand{\E}{{\mathbb E}}
\renewcommand{\P}{{\mathbb P}}
\newcommand{\F}{{\mathscr F}} 
 
\newcommand{\eps}{\varepsilon}

\newcommand{\G}{\mathcal{G}}

\newcommand{\Om}{\Omega} 
\newcommand{\om}{\omega}

\makeatletter
\newcommand{\plus}{%
  \DOTSB\mathop{\mathpalette\mattos@bigplus\relax}\slimits@
}
\newcommand\mattos@bigplus[2]{%
  \vcenter{\hbox{%
    \sbox\z@{$#1\sum$}%
    \resizebox{!}{0.9\dimexpr\ht\z@+\dp\z@}{\raisebox{\depth}{$\m@th#1+$}}%
  }}%
  \vphantom{\sum}%
}
\makeatother

\newcommand{\A}{{\mathcal A}}

\newcommand{\calL}{\mathcal{L}}

\newcommand{\dd}{\,\mathrm{d}}

\usepackage{stmaryrd}

\newcommand{\one}{{{\bf 1}}}

\newcommand{\into}{\hookrightarrow}
\newcommand{\ddd}{\mathrm{d}}

\allowdisplaybreaks

\begin{document}

\author{Michael Salins}
\address[Michael Salins]{Boston University\\United States of America.} 
\email{msalins@bu.edu} 
\author{Esm\'ee Theewis}
\address[Esm\'ee Theewis]{Delft University of Technology\\The Netherlands.}
\email{e.s.theewis@tudelft.nl}

\thanks{The first author was supported by Simons Foundation Grant 962543. The second author was supported by the VICI subsidy VI.C.212.027 of the Netherlands Organisation for Scientific Research (NWO)}   

\date{August 20, 2026}
 
\title[ULDP  for the stochastic heat equation]{Uniform large deviation principles for the stochastic heat equation  
over unbounded sets of initial data}



\begin{abstract}
We study small-noise large deviations for the stochastic heat equation (SHE) on the torus with unbounded, multiplicative space-time white noise. We establish uniform large deviation principles (ULDPs) over unbounded sets of initial data, alongside novel well-posedness and regularity results.  
The ULDPs are obtained for both  continuous-in-space  and  $L^p$ initial data. 
For these two classes, the ULDP holds uniformly over $L^q$-bounded subsets, with  $q<\infty$ and $q<p$ allowed respectively, so these sets  are highly unbounded in the state space. 
The admissible range of $q$ is dictated by the growth of the noise coefficient in the SHE. 
Crucially, our methods allow us to reach down to $q=1$. This yields ULDPs over $L^1$-bounded subsets, enabling the study of exit times in physical systems with mass conservation. 
\end{abstract}

\maketitle

\section*{Introduction}
 
We consider the following stochastic heat equation (SHE)  
on the one-dimensional torus $D\ceqq \R/2\pi\Z$, for small $\eps>0$:  
\begin{equation}\label{eq:SHE}
\begin{cases}
\frac{\partial X_x^{\eps}}{\partial t}(t, \xi) = \Delta X_x^{\eps}(t, \xi) +b(X_x^{\eps}(t,\xi))+ \sqrt{\eps} \sigma(X_x^{\eps}(t, \xi))\dot{W}(t,\xi), \quad  \xi \in D, t \in(0,T], \\ 
X_x^{\eps}(0, \xi) = x(\xi), \quad\xi\in D,
\end{cases} 
\end{equation} 
where $W$ is a space-time white noise.  
In this paper, we study uniform large deviation principles (ULDPs) for families of solutions $(X^\eps_x)_{\eps,x}$ to \eqref{eq:SHE}, as $\eps\to 0$.  We prove ULDPs in the $C([0,T];L^p(D))$ topology  for $p\geq 2$ and the $C([0,T]\times D)$ topology on the path space of solutions to \eqref{eq:SHE}. 
Unlike previous results for this problem, our main results prove that the ULDP holds uniformly over initial data  that are bounded in $L^q(D)$ norm, for some $q<p$. In this regime, the set $\{x \in L^p(D): \|x\|_{L^q(D)} \leq R\}$ is    unbounded in $L^p(D)$. Previous results mostly prove uniformity over bounded sets or compact sets of initial data \cite{CR04,BDM08}.  

A large deviation principle formally quantifies the  exponential decay rates for probabilities of rare events. In the context of the SHE, large deviation principles allow  us to accurately study the asymptotic behavior of solutions  $X_x^\eps$, for the small-noise limit $\eps\to 0$. Beyond guaranteeing almost sure convergence of $X_x^\eps$ to the deterministic solution $X_x^0$, they   provide the precise rates of decay for probabilities of relevant events. A powerful application of this theory is the Freidlin--Wentzell exit time problem, due to Freidlin and Wentzell \cite{FW3rdedition}. For $G\subset C(D)$ (or $G\subset L^p(D)$), using large deviation principles, one can characterize the exponential
growth rates of the corresponding exit times for the solutions $X_x^\eps$, defined as
\[
\tau^\eps_x\ceqq \inf\{t>0:X_x^\eps(t,\cdot)\notin G\}. 
\]
For these exit time applications, it is crucial to establish an LDP that holds uniformly with respect to the initial data in $G$, i.e.\ a ULDP (see Definition \ref{def:uldp}). We refer the reader to  \cite{FW3rdedition,DZ10} for details on the general proof strategy used in such applications. 

A highly influential development for proving ULDPs for small-noise SPDEs is the weak convergence approach by Budhiraja, Dupuis, and Maroulas \cite{BDM08}, which yields uniformity over compact sets of initial data. This approach has been applied in a large body of literature on LDPs \cite{LDP1,LDP2,LDP3,LDP4,LDP5,LDP6,LDP7,LDP8,LDP9,LDP10,LDP11,LDP12,LDP13,LDP14,LDP15,LDP16,LDP17,LDP18,LDP19,LDP20,LDP21,LDP22,LDP23} and ULDPs \cite{ULDP1,ULDP2,ULDP3,ULDP4}.  However, for the exit problem above, one typically wants to consider sets $G$ with  non-empty interior. Because such sets are non-compact in infinite-dimensional Banach spaces and one needs uniformity over initial data in $G$,   the weak convergence approach cannot be applied directly. 
In \cite{Salins19equivalences}, the first  author showed that the variational principle underlying \cite{BDM08} can nevertheless be used to establish uniform large deviation principles on non-compact, and even unbounded, sets of initial data, provided that the controlled equations satisfy a stronger mode of convergence than weak convergence.

Concerning the SHE and more general reaction-diffusion equations,  the earliest LDP results include \cite{freidlin88,Sowers92,KX96} and treat Lipschitz coefficients $b$ and $\sigma$.  
In \cite{CR04}, Cerrai and R\"ockner relaxed the assumptions of those works by removing Lipschitz
continuity from reaction terms and removing  ellipticity conditions on the multiplicative noise terms, and they were the first to prove a ULDP over supremum-norm-bounded subsets of initial data.  
More recently, the work \cite{Salins21reacdiff} by the first author further relaxed the assumptions of \cite{CR04}  by removing  local Lipschitz
 and polynomial growth conditions for the reaction terms, while establishing a  ULDP that still holds   uniformly over supremum-norm-bounded subsets of initial data, and even uniformly over all initial data in case $\sigma$ is bounded.

 The purpose of the current work is to establish, for \emph{unbounded} noise coefficients $\sigma$, a ULDP that holds uniformly over sets of initial data that are \emph{unbounded} in $C(D)$, i.e.\ sets for which we do not have pointwise-in-space control. 
More precisely, we will prove large deviation principles that hold uniformly over $L^q$-bounded subsets of initial data, where $q<\infty$. We assume that  $b$ and $\sigma$ are Lipschitz and grow (sub)linearly:  
\begin{assumption}\label{asm}
$T\in(0,\infty)$, $b,\sigma\col  \R\to \R$ are Lipschitz continuous, $\lambda\in(0,1]$, and there is a constant $C$ such that for all  $z\in \R$:
\begin{align}\label{eq:growth}
\begin{split}
  &|b(z)|\leq C(1+|z|)\quad \text{ and }\quad |\sigma(z)|\leq C(1+|z|^\lambda).
\end{split}
\end{align}
\end{assumption}
While we anticipate that our results generalize to multi-dimensional systems and non-Lipschitz dissipative drifts, our primary focus in this work is establishing uniformity with respect to unbounded sets.

Our first main result is the following. The precise notions of mild solution and ULDP will be discussed in Section \ref{sec:prelims}.  
 
\begin{theorem}[{ULDP on $C([0,T];C(D))$}]\label{th:C} Let Assumption \ref{asm} hold.  Suppose that $q\in[1,\infty)$ and  
\begin{equation*} 
2\lambda<q, 
\end{equation*}
where $\lambda$ is the growth rate of $\sigma$ from \eqref{asm}. 
Then, the family $\{X_x^{\eps}:\eps>0,x\in C(D)\}$ of mild solutions to \eqref{eq:SHE} satisfies the ULDP on $C([0,T];C(D))$, uniformly over initial data in $L^q(D)$-bounded subsets of $C(D)$, with rate functions $I_x\col C([0,T];C(D))\to [0,+\infty]$ given by 
\begin{align}\label{eq:rate}
\begin{split}
I_x(\varphi)\ceqq \frac12\inf\{\|u&\|_{L^2(0,T;L^2(D))}^2 :     u\in L^2(0,T;L^2(D)),\,\varphi=X_x^{0,u}\}, 
\end{split}
\end{align}
where $\inf\varnothing\ceqq +\infty$ and $X_x^{0,u}$ is the mild solution to 
\begin{equation*} 
\begin{cases}
\frac{\partial X}{\partial t}(t, \xi) = \Delta X(t, \xi) +b(X(t,\xi))+ \sigma(X(t, \xi))u(t,\xi), \quad  \xi \in D, t \in(0,T], \\ 
X(0, \xi) = x(\xi), \quad\xi\in D.
\end{cases} 
\end{equation*}  
\end{theorem}

\begin{remark}
If $\lambda\in(0,1/2)$, then  Theorem \ref{th:C} establishes the ULDP uniformly over $L^1(D)$-bounded subsets of initial data in $C(D)$, in the notably strong $C([0,T];C(D))$ topology. 
\end{remark} 

In many applications, the solutions to the SHE are nonnegative and the solutions $X^\eps_x(t,\xi)$ represent mass densities of diffusing molecules. In this setting, the total mass of the system at time $t$ is the $L^1$ norm $\int_D X^\eps_x(t,\xi)\dd \xi$. Thus, taking $q=1$, Theorem \ref{th:C}  shows that when $\lambda<1/2$, the solutions $X^\eps_x$ satisfy a ULDP in the uniform topology that is uniform over sets of initial data $x\in C(D)$ with bounded mass. This is a big improvement over previous results that could only prove a ULDP that was uniform over sets of densities bounded in the $L^\infty$ norm. Returning to the motivation discussed above, this improvement paves the way for deriving probabilistic exit time estimates for mass-conserving systems.

Besides initial data in $C(D)$, we will also consider initial data in $L^p(D)$ for $p\in[2,\infty)$ and study uniform large deviations in the  $C([0,T];L^p(D))$ topology.  
In this setting, one expects the ULDP to hold in the $C([0,T];L^p(D))$ topology  uniformly over $L^p(D)$-bounded sets of initial data, and consequently uniformly over $L^q(D)$-bounded sets  for $q\geq p$ (since $L^q(D)\into L^p(D)$). 
We cover this expected regime. The core novelty of the current work, however, lies in establishing the ULDP  for the  regime where $q<p$ and the $L^q(D)$ balls are unbounded in the $L^p(D)$ topology. Our result is as follows. 

\begin{theorem}[{ULDP on $C([0,T];L^p(D))$}]\label{th:Lp}  
Let Assumption \ref{asm} hold. Suppose that $p\in[2,\infty), q\in[1,\infty)$, and  
\begin{equation*} 
\frac{2\lambda p}{p+2}< q,  
\end{equation*}
where $\lambda$ is the growth rate of $\sigma$ from \eqref{asm}. 
Then, the family $\{X_x^{\eps}:\eps>0,x\in L^p(D)\}$ of mild solutions to \eqref{eq:SHE}  satisfies the ULDP on $C([0,T];L^p(D))$, uniformly over initial data in $L^q(D)$-bounded subsets of $L^p(D)$, with rate functions $I_x\col C([0,T];L^p(D))\to [0,+\infty]$ given by the formula \eqref{eq:rate}. 
\end{theorem} 
The lower bound for $q$ satisfies $\frac{2\lambda p}{p+2}<\lambda p\leq p$, thus indeed, Theorem \ref{th:Lp} yields an admissible range of $q<p$  for which we have uniform large deviations over $L^q(D)$-bounded -- but  $L^p(D)$-unbounded -- sets of initial data. In particular, we obtain uniformity over $L^1(D)$-bounded initial data for all $p$ if the growth rate $\lambda$ is at most $1/2$, and for a $\lambda$-dependent range of $p$ if $\lambda\in (1/2,1)$:

\begin{remark}
If $\lambda\in(0,1/2]$, then the lower bound for $q$ in Theorem \ref{th:Lp} is automatically satisfied for $q=1$ and any  $p\in[2,\infty)$.  
Thus we have a ULDP on $C([0,T];L^p(D))$ over $L^1(D)$-bounded subsets of initial data in $L^p(D)$ for all $p\in [2,\infty)$. 

If $\lambda\in(1/2,1)$, then we also have the ULDP on $C([0,T];L^p(D))$ with uniformity over $L^1(D)$-bounded subsets of $L^p(D)$, for all $p$ in the non-empty range $[2,\frac{2}{2\lambda-1})$. In particular,  $p=2$ is always admissible.
\end{remark}

We prove the ULDP theorems by studying the convergence of controlled mild solutions. We study
\begin{align*} 
  X_x^{\eps,u}(t,\xi) = [S(t)x](\xi)+V^{\eps,u}_x(t,\xi)+Y_x^{\eps,u}(t,\xi)+\sqrt{\eps}Z_x^{\eps,u}(t,\xi),
\end{align*}
where  
\begin{align*} 
\begin{split}
V^{\eps,u}_x(t,\xi) &=  
  \int_0^t \int_D G(t-s,\xi-\eta)b(X_x^{\eps,u}(s,\eta))\dd \eta\dd s, \\
  Y_x^{\eps,u}(t,\xi) &=  
  \int_0^t \int_D G(t-s,\xi-\eta)\sigma(X_x^{\eps,u}(s,\eta))u(s,\eta)\dd \eta\dd s, \\
  Z_x^{\eps,u}(t,\xi) &=   
  \int_0^t\int_D G(t-s,\xi-\eta)\sigma(X_x^{\eps,u}(s,\eta))W(\ddd\eta\dd s).\end{split}
\end{align*}
Here, $S(t)$ is the heat semigroup and $G(t,\xi)$ is the heat kernel defined in Subsection \ref{sub:heatkernel}. 
A key ingredient in our analysis is the semigroup regularization estimate  
\[\|S(t)x\|_{L^p(D)} \lesssim t^{-\frac{1}{2}(\frac{1}{q} - \frac{1}{p})}\|x\|_{L^q(D)}, \qquad 
 1\leq q< p \leq \infty
.\]
The above time singularity shows that $S(\cdot)x \in L^r([0,T];L^p(D))$ for all $r$ satisfying $\frac{r}{2}(\frac{1}{q} - \frac{1}{p})<1$.
We demonstrate that under the assumptions of Theorem  \ref{th:Lp} or\ \ref{th:C}, the integral terms $V^{\eps,u}_x, Y^{\eps,u}_x$, and $Z^{\eps,u}_x$ all belong to $C([0,T];L^p(D))$ (resp. $C([0,T]\times D)$), and establish important bounds that depend on the $L^q(D)$ norm of the initial data $x$. 
These calculations enable us to prove the ULDP and also yield novel existence and uniqueness results for rough initial data. The latter are presented in Theorems  \ref{th:wpnLp} and \ref{th:wpnC}. Smoothing properties of the SHE with rough initial data have previously been studied in  \cite{CD2014,CD2015,CH2019,CH2023,CCY2024,CK2019,FKN2025,SWZ2026,CX2026}. Our focus differs from  those results because we establish $L^p(D)$ estimates that are uniform over $L^q(D)$-bounded initial data, and we show how the growth rate $\lambda$ of $\sigma$ affects the regularization.

Building on our well-posedness results and the proofs of the ULDP theorems above, we establish in Section \ref{sec:additionalULDPs} additional ULDPs that hold  uniformly over \emph{arbitrary  bounded sets of $L^q(D)$-initial data}. That is, the sets of initial data are no longer assumed to be subsets of  $C(D)$ or $L^p(D)$. Because of this, the solutions no longer take  values in  $C([0,T];C(D))$ or $C([0,T];L^p(D))$, so the ULDP in these topologies is inherently ill defined. However, if we subtract off the irregular semigroup term, we can prove that $(X^\eps_x - S(\cdot)x)$ satisfies a ULDP in the $C([0,T];L^p(D))$ (resp.\ $C([0,T]\times D)$) topology, see Corollary \ref{cor:uldp pure Lq}. As a further consequence,  
we also establish ULDPs for $(X^\eps_x)$ in weaker topologies, such as  $L^{r}( 0,T;L^p(D))$-type topologies, while allowing for initial data belonging only to $L^q(D)$.

The organization of this paper is as follows. 
In Section \ref{sec:prelims}, we review the relevant background theory, specify the solution notion that we adopt, and state  sufficient conditions from \cite{Salins19equivalences} that we will use to prove our ULDP results. 
In Section \ref{sec:wpn estimates}, we establish deterministic and stochastic convolution estimates for the stochastic heat equation and for associated control problems that are relevant for the ULDP. Using these estimates, we establish new well-posedness and regularity results in Theorems \ref{th:wpnLp} and \ref{th:wpnC}.   
In Sections \ref{sec:C} and \ref{sec:Lp}, we prove our main ULDP results (Theorems \ref{th:C} and \ref{th:Lp}) using the estimates of Section \ref{sec:wpn estimates} together with an efficient Gr\"onwall scheme. In Section \ref{sec:additionalULDPs}, we derive further ULDPs based on these results. In particular, we establish ULDPs for initial data belonging only to $L^q(D)$, using alternative topologies.

\subsubsection*{Acknowledgement}
The authors are grateful to Mark Veraar for encouraging and supporting E.T.'s research visit to M.S.\ at Boston University, where this work was carried out. They also thank Boston University for hosting this visit.  

\section{Preliminaries}\label{sec:prelims}

The following notation will be used throughout the paper.
\begin{itemize}
\item The notation `$\lesssim$' means `less than or equal to a constant multiple of'. Subscripts, as in `$\lesssim_{T,p}$', indicate the parameters on which the implicit constant may depend. 
In statements of the form `$\lesssim$ ... a.s.', the constant is uniform with respect to a.e.\ $\om\in\Om$.

Although the measure $|D|=2\pi$ of the torus $D$ is fixed throughout the paper, we retain $|D|$ in such subscripts to emphasize when the finite measure of the domain is used.  
\item $*$ denotes the convolution operator over the domain $D$ for the spatial variable, or over $[0,T]$ for the time variable, as will be clear from the context.  
\item $\diamond$ denotes the stochastic convolution with respect to space-time white noise. 
\end{itemize}  

In this section, we discuss the    theoretical background that is relevant for our ULDP results. From now on, we fix an arbitrary filtered probability space $(\Om,\F,\P,(\F_t))$ satisfying the usual conditions. 

\subsection{Heat semigroup}\label{sub:heatkernel}

Let $G\col[0,\infty)\times D\to \R$ be the periodic heat kernel defined by 
\[
G(t,\xi)\ceqq  \frac{1}{2\pi} + \sum_{k=1}^{\infty} {\frac{1}{\pi}} e^{-|k|^2 t} \cos(k\xi),
\]
and define $S$ through 
\[
[S(t)x](\xi) \ceqq [G(t,\cdot)*x](\xi)=\int_D G(t,\xi-\eta)x(\eta)\dd \eta, \qquad t\in[0,\infty), \xi\in D.
\]
The kernel $G$ satisfies the following (see \cite[Lem.\ 2.1]{Salins25SHE}):  
\begin{align} 
  &\|G(t,\cdot)\|_{L^1(D)}=1,\label{eq:L1}\\
  &\|G(t,\cdot)\|_{L^\infty(D)}\leq C(1+t^{-1/2})\leq C(1+T^{1/2})t^{-1/2} \text{ for all $T\in(0,\infty)$ and $t\in(0,T]$,} \label{eq:Linfty}
\end{align}where $C$ is a constant independent of $T$ and $t$. 

As a consequence of \eqref{eq:L1} and Young's convolution inequality, $S(t)\in \calL(L^p)$ for all $p\in [1,\infty]$. Moreover, one can show that $S$ is a strongly continuous semigroup on $L^p(D)$, and on $C(D)$. For all these domains, we call the resulting semigroup the \emph{heat semigroup} and denote it simply by $S$. 

The kernel bounds above and  interpolation (or H\"older's inequality) yield for all $r\in[1,\infty]$:
\begin{equation}\label{eq:G Lr norm}
    \|G(t,\cdot)\|_{L^r(D)}\leq \|G(t,\cdot)\|_{L^\infty(D)}^{1-\frac1r}\|G(t,\cdot)\|_{L^1(D)}^{\frac1r}
    \lesssim_r (1+T^{\frac12(1-\frac1r)})t^{-\frac12(1-\frac1r)}.
\end{equation}
Combining \eqref{eq:G Lr norm} with Young's convolution inequality, the following smoothing estimate holds for the heat semigroup, for any $T\in(0,\infty)$, $t\in(0,T]$, $1\leq q\leq p_1\leq \infty$ and $x\in L^q(D)$: 
\begin{align}
  \|S(t)x\|_{L^{p_1}(D)} =\|G(t,\cdot)*x\|_{L^{p_1}(D)}  
\leq \|G(t,\cdot)\|_{L^r(D)}\|x\|_{L^q(D)} 
\leq  C_{T,q,p_1} t^{-\frac12(\frac1q-\frac{1}{p_1})}\|x\|_{L^q(D)}, \label{eq:est}
\end{align}
where $r\in[1,\infty]$ is such that $1+\frac{1}{p_1}=\frac1r+\frac1q$.  

Finally, using the semigroup property of $S$ and the fundamental lemma of the calculus of variations for instance, one can show that the kernel $G$ satisfies for all $0\leq s\leq t<\infty$ and  $\xi\in D$: 
\begin{equation}\label{eq:kernel semigrp prop}
  \int_D G(t-s,\xi-\eta)G(s,\eta) \dd\eta = G(t,\xi). 
\end{equation}

\subsection{Space-time white noise}\label{sub:white noise}

We will use both Walsh's definition and the functional analytic definition of space-time white noise. 

Let $T>0$, let $(e_k)$ be an orthonormal basis for $L^2(D)$ and let $(B_k)$ be a sequence of independent standard $(\F_t)$-Brownian motions. 
For any progressively measurable process $\varphi\col \Om\times [0,T]\times D\to \R$ such that $\varphi\in L^2(\Om\times[0,T]\times D)$, we define 
\begin{equation}\label{eq:Walsh int}
\int_0^t\int_D \varphi(s,\eta)W(\ddd\eta\dd s)\ceqq \sum_{k\in \N}\int_0^t\Big(\int_D\varphi(s,\eta)e_k(\eta)\dd \eta\Big)\dd B_k(s), \qquad t\in[0,T],
\end{equation}
where the $\ddd B_k(s)$ integral denotes the It\^o integral. We call $W$ the space-time white noise measure. 
 
We can identify $W$ with an $L^2(D)$-cylindrical Brownian motion $W\in \calL(L^2(0,T;L^2(D)),L^2(\Om))$, where the latter satisfies $W(\one_{(0,t]}\otimes e_k)= B_k(t)$ for all $t\in(0,\infty)$ and $k\in\N$. 

Let $p,\tilde p\in[2,\infty)$ and let $\varphi\col \Om\times[0,T]\times D \to L^p(D)$ be a progressively measurable process such that $\varphi \in L^{\tilde p}(\Omega; L^p(D; L^2(0,T; L^2(D))))$, identifying $\varphi(\om,s,\eta)(\xi)=\varphi(\om)(\xi)(s)(\eta)$.  
We can define the Walsh integral $\int_0^t\int_D\varphi(s,\eta)(\xi)W(\ddd \eta\dd s)$ pointwise in $\xi$ through \eqref{eq:Walsh int} for $t\in[0,T]$. 
Then, the following identity holds a.s.\ in $L^p(D)$ (with respect to $\xi$): 
\[
\int_0^t\int_D\varphi(s,\eta)(\xi)W(\ddd \eta\dd s) =  
\Big[\int_0^t \Phi\dd W\Big](\xi),
\]
where the $\ddd W$ integral is the stochastic integral in the framework of UMD spaces (see \cite{NVW15}), and    
\[
[\Phi(\omega)g](\xi)
\coloneqq
\int_0^T \int_D
\varphi(\om,s,\eta)(\xi)\,g(s,\eta)
\dd \eta\dd s,
\qquad
g\in L^2(0,T;L^2(D)).
\]
By the $\gamma$-Fubini isomorphism (see \cite[(3.1), (3.2)]{NVW15}), $\Phi$ is a random operator such that
\[
\Phi(\omega)
\in
\gamma\bigl(L^2(0,T;L^2(D)),L^p(D)\bigr) 
\] 
for a.e.\ $\om\in\Om$. 
Moreover,
\[
\|\Phi(\omega)\|_{\gamma(L^2(0,T;L^2(D)),L^p(D))}
\simeq_p
\bigg\|
\Big(
\int_0^T \int_D
|\varphi(\om,s,\eta)(\cdot)|^2
\dd \eta\dd s
\Big)^{1/2}
\bigg\|_{L^p(D)}.
\]

The BDG inequality in the  UMD Banach space $L^p(D)$ (see \cite[(5.5)]{NVW15}) and the isomorphism above yield  
\begin{equation}\label{eq:BDG Lp}
 \E \Big[ \sup_{t \in [0,T]} \Big\| \int_0^t \Phi(s)\mathrm{d}W(s) \Big\|_{L^p(D)}^{\tilde p} \Big] \simeq_{p,\tilde p} \E\Big[ \Big(\int_D \Big( \int_0^T \int_D |\varphi(s, \eta)(\xi)|^2 \dd\eta \dd s \Big)^{\frac{p}{2}} \dd \xi\Big)^{\frac{\tilde p}{p}} \Big].
\end{equation}

\subsection{Solution notion}

For $N\in[0,\infty)$, we define
\begin{align}\label{eq:AN}
\begin{split}
\A_N\ceqq \big\{u\in& L^2(\Om\times [0,T]\times D): \\
&\text{ $u$ is progressively measurable and } \P\big(\|u\|_{L^2(0,T;L^2(D))}\leq N\big)=1\big\}. 
\end{split}
\end{align}

For $u\in \A_N$ and $\eps\in[0,\infty)$, we consider the following implicit equation 
\begin{align}\label{eq:split X}
  X_x^{\eps,u}(t,\xi) = [S(t)x](\xi)+V^{\eps,u}_x(t,\xi)+Y_x^{\eps,u}(t,\xi)+\sqrt{\eps}Z_x^{\eps,u}(t,\xi),
\end{align}
where  
\begin{align}\label{eq:split Y Z}
\begin{split}
V^{\eps,u}_x(t,\xi) &=  
  \int_0^t \int_D G(t-s,\xi-\eta)b(X_x^{\eps,u}(s,\eta))\dd \eta\dd s, \\
  Y_x^{\eps,u}(t,\xi) &=  
  \int_0^t \int_D G(t-s,\xi-\eta)\sigma(X_x^{\eps,u}(s,\eta))u(s,\eta)\dd \eta\dd s, \\
  Z_x^{\eps,u}(t,\xi) &=   
  \int_0^t\int_D G(t-s,\xi-\eta)\sigma(X_x^{\eps,u}(s,\eta))W(\ddd\eta\dd s). \end{split}
\end{align}
Here, the $\ddd \eta$ and $\ddd s$ integrals are Lebesgue integrals, and the $W(\ddd\eta\dd s)$ integral is the stochastic integral with respect to the space-time white noise $W$ as defined in Subsection \ref{sub:white noise}. 

Observe that if $u=0$, then the equation for $X_x^{\eps,0}$ corresponds to the mild  formulation of the original stochastic heat equation \eqref{eq:SHE} discussed in the introduction. Furthermore, if $\eps=0$ and $u\in L^2(0,T;L^2(D))$ (so $u$ is $\om$-independent), then the equation for $X_x^{0,u}$ corresponds to  the so-called \emph{skeleton equation} that appears in the rate functions $I_x$ in Theorems \ref{th:C} and \ref{th:Lp} (see \eqref{eq:rate}). 

We define solutions in our framework as follows. 
 
\begin{definition}[Mild solution]\label{def:sol}
Let $\eps,N\in[0,\infty)$, $u\in \A_N$, $x\in L^1(D)$ and let $X_x^{\eps,u}\col \Om\times [0,T]\times D\to \R$ be a progressively measurable  process. 
If a.s., \eqref{eq:split X} and \eqref{eq:split Y Z} are  satisfied for Lebesgue-a.e.\ $t\in[0,T]$ and $\xi\in D$ (in particular, the integrals appearing therein are well-defined), then we call $X_x^{\eps,u}$ a \emph{mild solution to \eqref{eq:split X}}. 

If $u=0$, and $X_x^{\eps,0}$ is a mild solution to \eqref{eq:split X} in the sense above, then we call $X_x^\eps\ceqq X_x^{\eps,0}$ a mild solution to \eqref{eq:SHE}.
\end{definition}

\subsection{ULDP and sufficient criteria}
 
The Freidlin--Wentzell uniform large deviation principle (with speed $\eps$) is defined as follows. 
\begin{definition}[ULDP]\label{def:uldp}
Let $\mathcal{E}$ be a Polish space, completely metrized by a metric $d$. 
Let $(\Omega, \mathcal{F}, \mathbb{P})$ be a probability space.  Let $ \mathcal{E}_0$ be any set, and let $\{X^\eps_x: \eps>0,x\in \mathcal{E}_0\}$ be a family of $\mathcal{E}$-valued random variables. Let $\{I_x : x \in \mathcal{E}_0\}$ be a collection of lower-semicontinuous maps $I_x : \mathcal{E} \to [0,+\infty]$. 
Let $\mathscr{A}$ be a collection of subsets of $\mathcal{E}_0$.  
Then, the family $\{X^\eps_x: \eps>0,x\in \mathcal{E}_0\}$ is said to satisfy a \emph{Freidlin--Wentzell uniform large deviation principle (ULDP) with respect to the rate functions $I_x$ uniformly over $\mathscr{A}$}, if
\begin{itemize}
\item For any $A \in \mathscr{A}$, $s_0 > 0$, and $\delta > 0$,
\[
\liminf_{\eps \to 0} \inf_{x \in A} \inf_{\varphi \in \Phi_x(s_0)} \big(\eps\log \mathbb{P}(d(X^\eps_x, \varphi) < \delta) + I_x(\varphi)\big) \geq 0.
\]
\item For any $A \in \mathscr{A}$, $s_0 > 0$, and $\delta > 0$,
\[
\limsup_{\eps \to 0} \sup_{x \in A} \sup_{s \in [0,s_0]} \big(\eps \log \mathbb{P}(\mathrm{dist}_d(X^\eps_x, \Phi_x(s)) \geq \delta) + s\big) \leq 0,
\]
\end{itemize}
where $\Phi_x(s) \ceqq \{\varphi \in \mathcal{E} : I_x(\varphi) \leq s\}$. 
\end{definition}

From now on, we specialize to $X^\eps_x$ being a solution to the stochastic heat equation \eqref{eq:SHE}, and we will let $\mathcal{E}_0= C(D)$, $\mathcal{E}=C([0,T];C(D))$, or $\mathcal{E}_0= L^p(D)$, $\mathcal{E}=C([0,T];L^p(D))$, respectively. 
In these settings, we will consider large deviation principles that hold uniformly over  families of $L^q$-bounded subsets of $C(D)$ and $L^p(D)$. For  $q\in[1,\infty)$ and $R\in[0,\infty)$, we define 
\begin{align}\label{eq:BqR}
 B_{R}^{q}\ceqq \{x\in L^q(D) :\|x\|_{L^q(D)}\leq R\}. 
\end{align} 
 
The next result follows from \cite[Th.\ 2.13]{Salins19equivalences}. It provides sufficient criteria for the ULDPs stated in Theorems \ref{th:C} and  \ref{th:Lp}, and Corollary \ref{cor:Lp p<2}. We will use these criteria to prove the latter theorems. Here, we denote by $X_x^{\eps,u}$  the unique mild solution that will be provided by Corollary  \ref{cor:wpn classical}.  

\begin{theorem}[{Special case of \cite[Th.\ 2.13]{Salins19equivalences}}]
\label{th:uldp suff}
Let Assumption \ref{asm} hold. Then:
\begin{enumerate}
  \item If $p,q,$ and $\lambda$ are as in Theorem \ref{th:C},   and for all $N,R,\delta\in( 0,\infty)$, 
\[
\lim_{\eps\downarrow 0}\sup_{x\in B_{R}^q \cap C(D)}\sup_{u\in\A_N}\P(\|X_x^{\eps,u}-X_x^{0,u}\|_{C([0,T];\Cper)}>\delta)=0, 
\]
then the family $\{X_x^{\eps,0}:\eps>0,x\in C(D)\}$ satisfies the ULDP uniformly over $\{B_{R}^q\cap C(D):R\in[0,\infty)\}$,  with the rate functions $I_x$  stated in Theorem \ref{th:C}. 
\item  If $p,q,$ and $\lambda$ are as in Theorem \ref{th:Lp} or Corollary  \ref{cor:Lp p<2}, and for all $N,R,\delta\in( 0,\infty)$,  
    \[
\lim_{\eps\downarrow 0}\sup_{x\in B_R^{q}\cap L^p(D)}\sup_{u\in\A_N}\P(\|X_x^{\eps,u}-X_x^{0,u}\|_{C([0,T];L^p(D))}>\delta)=0,  
\] 
then the family $\{X_x^{\eps,0}:\eps>0,x\in L^p(D)\}$ satisfies the ULDP uniformly over $\{B_{R}^{q}\cap L^p(D):R\in[0,\infty) \}$, with the rate functions $I_x$  stated in Theorem \ref{th:Lp}. 
\end{enumerate}
\end{theorem}

\section{Convolution estimates, existence, uniqueness, and regularity}\label{sec:wpn estimates}

First, in Subsection \ref{sub:convolution estimates}, we prove  deterministic and stochastic convolution results for the heat semigroup, which will then be used for our well-posedness results in Subsection \ref{sub:wpn} and are also essential for the ULDP proofs later on. 

\subsection{Deterministic and stochastic convolution estimates}\label{sub:convolution estimates}

The next lemma demonstrates a regularizing effect -- specifically, an improvement in spatial integrability -- for the mapping $\varphi u\mapsto S*(\varphi u)$. This estimate will serve as a crucial tool for bounding terms $Y_x^{\eps,u}$ and $Y_x^{0,u}$ in our subsequent proofs.

\begin{lemma} 
\label{lem:det conv}
Let $T,N\in(0,\infty)$  and suppose that  
$p_1\in(2,\infty]$, $p_2\in[p_1,\infty]$, or  $p_1=2$,  $p_2 \in [2,\infty)$.   
Let $u\in L^2(0,T;L^2(D))$ be such that $\|u\|_{L^2(0,T;L^2(D))}\leq N$. Let 
\begin{equation}\label{eq:cond tilde p}
    \tilde p\in\Big(\frac{4}{1-\frac{2}{p_1}+\frac{2}{p_2}},\infty\Big)
\end{equation}
and let $\varphi\in L^{\tilde p}(0,T;L^{p_1}(D))$. 
Define 
\[
Y(t,\xi)\ceqq \big[S*[\varphi(\cdot)u(\cdot)]\big](t)(\xi)=\int_0^t \int_{D} G(t-s,\xi-\eta)\varphi(s,\eta)u(s,\eta)d\eta \dd s.\]
Then, it holds that $Y\in C([0,T];L^{p_2}(D))$,   and for any   $t\in [0,T]$,
\[
\|Y(t,\cdot)\|_{L^{p_2}(D)}^{\tilde p}  \lesssim_{{\tilde p}} N^{{\tilde p}}(1+T^{\frac{\tilde p}{2}-1})\int_0^t\|\varphi(s,\cdot)\|_{L^{p_1}(D)}^{{\tilde p}}\dd s.
\]
Moreover, if $p_2=\infty$, then $Y\in C([0,T];C(D))$. 
\end{lemma}  
\begin{proof}

First, by the Cauchy--Schwarz inequality in $L^2(0,T;L^2(D))$, we have for all $t\in[0,T]$ and $\xi\in D$:
\[
|Y(t,\xi)|\leq\int_0^t \int_D |G(t-s,\xi-\eta)\varphi(s,\eta)u(s,\eta)|\dd\eta \dd s
\leq N\Big(\int_0^t \int_D G^2(t-s,\xi-\eta)|\varphi(s,\eta)|^2\dd\eta \dd s\Big)^{\frac12}.
\]

If $p_1\in[2,\infty)$, then we apply Minkowski's inequality, followed by Young's convolution inequality (with powers $\frac{p_2}{2}, \frac{p_1}{2},$ and 
$\rho\ceqq \frac{1}{1+\frac{2}{p_2}-\frac{2}{p_1}}$) and the bound \eqref{eq:G Lr norm} for $G$, giving
\begin{align*}
   \|Y(t,\cdot)\|_{L^{p_2}(D)}^2  &=\|Y(t,\cdot)^2\|_{L^{\frac{p_2}{2}}(D)}\\
   &\leq N^2 \Big\|\xi\mapsto\int_0^t  \big[G^2(t-s,\cdot)*|\varphi(s,\cdot)|^2\big](\xi)  \dd s\Big\|_{L^{\frac{p_2}{2}}(D)} \\
   &\leq N^2  \int_0^t \|G^2(t-s,\cdot)*|\varphi(s,\cdot)|^2\|_{L^{\frac{p_2}{2}}(D)}\dd s \\
   &\leq  N^2 \int_0^t   \|G(t-s,\cdot)\|_{L^{2\rho}(D)}^2\||\varphi(s,\cdot)|^2\|_{L^{\frac{p_1}{2}}(D)}\dd s\\ 
   &\lesssim N^2 (1+T^{\frac12+\frac{1}{p_1}-\frac{1}{p_2}}) \int_0^t (t-s)^{-\frac12-\frac{1}{p_1}+\frac{1}{p_2}} \|\varphi(s,\cdot)\|_{L^{p_1}(D)}^{2 } \dd s. 
\end{align*} 
Note that $-\frac12-\frac{1}{p_1}+\frac{1}{p_2}>-1$ by the assumed ranges of $p_1$ and $p_2$. 

Consequently, putting $\theta\ceqq \frac12+\frac{1}{p_1}-\frac{1}{p_2}$ and applying H\"older's inequality with $r\ceqq \frac{\tilde p}{\tilde p-2}, r'=\frac{\tilde p}{2}$, we have  for all  $t\in[0,T]$:
\begin{align}
  \|Y(t,\cdot)\|_{L^{p_2}(D)}^{{\tilde p}}   
&\lesssim N^{{\tilde p}}(1+T^{\theta})^{\frac{\tilde p}{2}} \Big(\int_0^t(t-s)^{-\theta}\|\varphi(s,\cdot)\|_{L^{p_1}(D)}^2\dd s\Big)^{\frac{\tilde p}{2}} \notag\\ 
  &\leq N^{{\tilde p}}(1+T^{\theta})^{\frac{\tilde p}{2}} \Big(\int_0^t(t-s)^{-\theta r}\dd s\Big)^{\frac{1}{r}\frac{\tilde p}{2}} \Big(\int_0^t\|\varphi(s,\cdot)\|_{L^{p_1}(D)}^{\tilde p}\dd s\Big)  \notag\\ 
 &\lesssim_{{\tilde p}} N^{{\tilde p}}(1+T^{\theta})^{\frac{\tilde p}{2}}T^{\frac{\tilde p(1-\theta)-2}{2}}\int_0^t\|\varphi(s,\cdot)\|_{L^{p_1}(D)}^{{\tilde p}}\dd s\notag\\
 &\lesssim_{\tilde p} N^{{\tilde p}}(1+T^{\frac{\tilde p}{2}-1})\int_0^t\|\varphi(s,\cdot)\|_{L^{p_1}(D)}^{{\tilde p}}\dd s, \label{eq:Yestreg}
\end{align} 
where we used that $- \theta r>-1$ if and only if  
$$
\tilde p>\frac{2}{1-\theta}=\frac{4}{1-\frac{2}{p_1}+\frac{2}{p_2}},
$$
i.e.\ if and only if \eqref{eq:cond tilde p} holds. 

If $p_1=\infty$, then $p_2=\infty$ and the above still applies with $\frac{\infty}{2}=\infty$ and $\rho=1$. 
Thus, in either case, we conclude that $Y\in L^\infty(0,T;L^{p_2}(D))$ and the claimed bound holds. 
 
Next, we prove that $Y\in C([0,T];L^{p_2}(D))$ if $p_2<\infty$, by approximating $u$ and $\varphi$. 
Assume that $u\in L^\infty([0,T]\times D)$ and $\varphi\in L^{\tilde p}(0,T;L^{p_2}(D))$. 
Then, $\varphi(\cdot)u(\cdot)\in L^{\tilde p}(0,T;L^{p_2}(D))$ and $S$ is a strongly continuous semigroup on $L^{p_2}(D)$, so by classical semigroup theory we have $Y=S*(\varphi u)\in C([0,T];L^{p_2}(D))$. Consequently, the following bilinear operator is well-defined:
\[
\Psi \col L^\infty([0,T]\times D)\times L^{\tilde p}(0,T;L^{p_2}(D))\to C([0,T];L^{p_2}(D))\col (u,\varphi)\mapsto S*(\varphi u).
\]
By the bound \eqref{eq:Yestreg} (take $N=\|u\|_{L^2(0,T;L^2(D))}$),  we have 
\begin{equation}\label{eq:skel conv cont}
\|\Psi (u,\varphi)\|_{C([0,T];L^{p_2}(D))}\lesssim_{\tilde p,T} \|u\|_{L^2(0,T;L^2(D))}\|\varphi\|_{L^{\tilde p}(0,T;L^{p_1}(D))}.
\end{equation}
Thus, by density of  $L^\infty([0,T]\times D)$ in $L^2(0,T;L^2(D))$ and $L^{\tilde p}(0,T;L^{p_2}(D))$ in $L^{\tilde p}(0,T;L^{p_1}(D))$, $\Psi$ extends to a bounded bilinear operator on $L^2(0,T;L^2(D))\times L^{\tilde p}(0,T;L^{p_1}(D))$, proving the claimed continuity for the case $p_2<\infty$. 

It remains to prove that $Y \in C([0,T]; C(D))$ when $p_2 = \infty$. Since $D$ is  bounded, we may assume without loss of generality that $p_1 < \infty$ while \eqref{eq:cond tilde p} is still satisfied and $\varphi\in L^{\tilde p}(0,T;L^{p_1}(D))$. 
For $u,\varphi \in C([0,T]; C(D))$, we have $\Psi(u,\varphi)= S*(\varphi u) \in C([0,T];C(D))$ by classical strongly continuous semigroup theory. 
Then, from  \eqref{eq:Yestreg} and density of  $C([0,T];C(D))$  in $L^2(0,T; L^2(D))$ and $L^{\tilde p}(0,T; L^{p_1}(D))$, it follows that $\Psi$ extends uniquely to a bounded bilinear map $L^2(0,T;L^2(D))\times L^{\tilde p}(0,T;L^{p_1}(D))\to C([0,T]; C(D))$, completing the proof. 
\end{proof}

Next, we prove a $C([0,T];L^p(D))$-norm moment estimate for stochastic convolutions with the heat semigroup. Here, we denote by $\diamond$ the stochastic convolution in time with respect to the space-time white noise $W$, viewed as an $L^2(D)$-cylindrical Brownian motion (see Subsection \ref{sub:white noise}). 

\begin{proposition}\label{prop:stoch conv}  
Let $T\in(0,\infty)$, $p\in[2,\infty)$ and $\tilde p\in(4,\infty)$. Let $\Phi\col (0,T)\times \Om\to L^p(D)$ be a progressively measurable process such that 
$\E[\int_0^{T}\|\Phi(s)\|_{L^p(D)}^{\tilde p}\dd s]<\infty$. 
Let $$Z(t)\ceqq [S\diamond\Phi](t)= \textstyle{\int_0^t} S(t-s)\Phi(s)\dd W(s),$$ 
where $S$ is the heat semigroup on $L^p(D)$. 
Then, $Z\in L^{\tilde p}(\Om;C([0,T];L^p(D)))$,  and for all $T_1\in(0,T]$,  
\[
\E\big[\|Z\|_{C([0,T_1];L^{p}(D))}^{{\tilde p}}\big]\lesssim_{p,\tilde p}(1+T^{\frac{\tilde p}{4}})T^{\frac{\tilde p}{4}-1}\E\Big[\int_0^{T_1}\|\Phi(s)\|_{L^p(D)}^{\tilde p}\dd s\Big]. 
\] 
\end{proposition}
\begin{proof} 
It suffices to prove the bound with $T_1=T$, since one can afterwards apply the bound to $\Phi\one_{[0,T_1]}$ for $T_1<T$.  
Furthermore, it suffices to prove the bound for pointwise bounded $\Phi$.  Indeed, $\Phi\mapsto S*\Phi\col L^{\tilde p}(\Om ;L^{\tilde p}(0,T_1;L^p(D)))\to L^{\tilde p}(\Om;C([0,T_1];L^{p}(D)))$ is a linear operator, and pointwise bounded, progressively measurable processes are dense in the class of progressively measurable processes in  $L^{\tilde p}(\Om ;L^{\tilde p}(0,T_1;L^p(D)))$. Frow now on, we assume that $\Phi$ is pointwise bounded. 

We employ the factorization method from \cite{daprato}. We fix an arbitrary $\alpha\in(1/\tilde p, 1/4)$, recalling that $\tilde p>4$.  By \eqref{eq:L1} and Young's convolution inequality, we have 
$\|S(t)x\|_{L^p(D)}\leq \|x\|_{L^p(D)}.$  
Therefore, since ${\tilde p}>1$ and $\alpha>{1}/{\tilde p}$,   \cite[Prop.\ 5.9]{daprato}   
yields that 
\begin{equation}\label{eq:Galpha bnd}
G_\alpha\in \calL\big(L^{{\tilde p}}(0,T;L^p(D)),C([0,T];L^p(D))\big),\quad \|G_{\alpha}\|\lesssim_{\tilde p,\alpha}  T^{ \alpha-\frac{1}{\tilde p}},
\end{equation} 
where $G_\alpha(h)(t)\ceqq \int_0^t(t-s)^{\alpha-1}S(t-s)h(s)\dd s$. 
Furthermore, if $\alpha\in(0,1)$, then a.s.
\[
Z=S\diamond \Phi
=\tfrac{\sin(\alpha\pi)}{\pi} G_\alpha\big((\cdot)^{-\alpha}S\diamond \Phi\big)
=\tfrac{\sin(\alpha\pi)}{\pi} G_\alpha(Z_\alpha),
\]
where $\diamond$ denotes the stochastic convolution with respect to the space-time white noise $w$, and 
$$Z_\alpha(t,\xi)\ceqq [[((\cdot)^{-\alpha}S)\diamond \Phi](t)](\xi). $$ 
Combining  the formula for $S\diamond \Phi$ above with \eqref{eq:Galpha bnd}, we have 
\begin{equation}\label{eq:factor est}
\|S\diamond \Phi\|_{L^{\tilde p}(\Om;C([0,T];L^p(D)))}^{\tilde p}\lesssim_{\tilde p,\alpha}    T^{\tilde p \alpha-1}\|Z_\alpha\|_{L^{\tilde p}(\Om;L^{{\tilde p}}(0,T;L^p(D)))}^{\tilde p}. 
\end{equation}
Now, we first apply the BDG inequality  in the UMD space $L^p(D)$ (see \eqref{eq:BDG Lp}) for fixed $t\in[0,T]$. After that, we apply the bound \eqref{eq:Linfty} for $G$, followed by Minkowski's inequality (using that $p\in [2,\infty)$) and Young's convolution inequality in space, giving 
\begin{align*}
    \|Z_\alpha(t,\cdot)\|_{L^{\tilde p}(\Om;L^p(D))}^{\tilde p}
    &\lesssim_{p,\tilde p} \E\bigg[\Big\|\Big(\int_0^t\int_D (t-s)^{-2\alpha}G^2(t-s,\cdot-\eta)|\Phi(s,\eta)|^2\dd\eta\dd s \Big)^{\frac12}\Big\|_{L^p(D)}^{\tilde p} \bigg]\\
&= \E\bigg[\Big\| \int_0^t\int_D (t-s)^{-2\alpha}G^2(t-s,\cdot-\eta)|\Phi(s,\eta)|^2\dd\eta\dd s \Big\|_{L^{\frac{p}{2}}(D)}^{\frac{\tilde p}{2}} \bigg]\\
&\lesssim (1+T^{\frac12})^{\frac{\tilde p}{2}} \E\bigg[\Big\| \int_0^t (t-s)^{-2\alpha-\frac12}[G (t-s,\cdot)*|\Phi(s,\cdot)|^2] \dd s \Big\|_{L^{\frac{p}{2}}(D)}^{\frac{\tilde p}{2}} \bigg]\\
&\lesssim_{\tilde p} (1+T^{\frac{\tilde p}{4}}) \E\bigg[\Big(  \int_0^t (t-s)^{-2\alpha-\frac12}\big\|G (t-s,\cdot)*|\Phi(s,\cdot)|^2\big\|_{L^{\frac{p}{2}}(D)} \dd s \Big)^{\frac{\tilde p}{2}} \bigg]\\
&\leq (1+T^{\frac{\tilde p}{4}})\E\bigg[\Big(  \int_0^t (t-s)^{-2\alpha-\frac12}\big\|G (t-s,\cdot)\|_{L^1(D)}\||\Phi(s,\cdot)|^2 \|_{L^{\frac{p}{2}}(D)} \dd s \Big)^{\frac{\tilde p}{2}} \bigg]\\
&\lesssim (1+T^{\frac{\tilde p}{4}})\E\bigg[\Big(  \int_0^t (t-s)^{-2\alpha-\frac12} \|\Phi(s,\cdot) \|_{L^{p}(D)}^2 \dd s \Big)^{\frac{\tilde p}{2}} \bigg].
\end{align*}  
Applying  Tonelli's theorem twice, and applying the above, we obtain
\begin{align*}
    \| Z_\alpha\|_{L^{\tilde p}(\Om ;L^{{\tilde p}}(0,T;L^p(D)))}^{{\tilde p}} &= \| \|Z_\alpha \|_{L^{\tilde p}(\Om;L^p(D))}^{\tilde p}\|_{L^{1}(0,T)} \\
    &\lesssim_{p, \tilde p }(1+T^{\frac{\tilde p}{4}}) \int_0^T  \E\bigg[\Big(  \int_0^t (t-s)^{-2\alpha-\frac12} \|\Phi(s,\cdot) \|_{L^{p}(D)}^2 \dd s \Big)^{\frac{\tilde p}{2}} \bigg] \dd t\\
    &= (1+T^{\frac{\tilde p}{4}}) \E\bigg[\int_0^T\Big(  \int_0^t (t-s)^{-2\alpha-\frac12} \|\Phi(s,\cdot) \|_{L^{p}(D)}^2 \dd s \Big)^{\frac{\tilde p}{2}}\dd t \bigg] \\
    &= (1+T^{\frac{\tilde p}{4}})\E\Big[\big\| (\cdot)^{-2\alpha-\frac12}* \|\Phi \|_{L^{p}(D)}^2 \big\|_{L^{\frac{\tilde p}{2}}(0,T)}^ {\frac{\tilde p}{2}}\Big]\\
    &\leq(1+T^{\frac{\tilde p}{4}}) \E\big[ \|(\cdot)^{-2\alpha-\frac12}\|_{L^1(0,T)}^{\frac{\tilde p}{2}} \|\Phi \|_{L^{ {\tilde p} }(0,T;L^p(D))}^{\tilde p }\big]\\
    &\lesssim_\alpha(1+T^{\frac{\tilde p}{4}}) T^{(-2\alpha+\frac12)\frac{\tilde p}{2}}\E\big[\|\Phi \|_{L^{\tilde p}(0,T;L^p(D))}^{\tilde p}\big].
\end{align*}  
In the second-last line, we applied Young's convolution inequality in time with $1+\frac{1}{\tilde p/2}=1+\frac{1}{\tilde p/2}$. In the last line, we used that $\alpha<1/4$. 
Combining the last estimate with \eqref{eq:factor est}, we conclude that
\begin{align*}
\E[\|Z\|_{C([0,T];L^p(D))}^{\tilde p}]
&\lesssim_{\tilde p,\alpha} T^{\tilde p \alpha-1}\E[\|Z_\alpha\|_{L^{\tilde p}([0,T];L^p(D))}^{\tilde p}]
\lesssim_{p,\tilde p,\alpha} (1+T^{\frac{\tilde p}{4}})T^{ \frac{\tilde p}{4}-1}\E\big[\|\Phi \|_{L^{\tilde p}(0,T;L^p(D))}^{\tilde p}\big].
\end{align*}
Choosing, for instance, $\alpha=\frac12(\frac{1}{\tilde p}+\frac14)$, the implicit constant depends only on $p$ and $\tilde p$. 
\end{proof}

The following estimate for the stochastic convolution, uniform in both time and space, is due to \cite[Th.\ 1.2]{Salins25SHE}. 
The continuity of $Z$ in time and space  follows from
the proof of \cite[Th.\ 1.2]{Salins25SHE} and the factorization method    
\cite[Prop.\ 5.9, Th.\ 5.10]{daprato} applied with $E_1=C(D)$, $E_2=L^{\tilde p}(D)$, $p=\tilde p$, and $r=1/(2\tilde p)$. 

\begin{theorem}[{\cite[Th.\ 1.2]{Salins25SHE}}]\label{th:stoch conv C}
    Let ${\tilde p}\in (6,\infty)$. Let $\Phi\col \Om\times [0,T]\to L^{\tilde p}(D)$ be a progressively measurable process  such that
$\E\big[\int_0^T \|\Phi(t,\cdot)\|_{L^{\tilde p}(D)}^{\tilde p}\dd t\big] <  \infty$. 
Let
\begin{equation*}
Z(t,\cdot)\ceqq [S\diamond \Phi](t)=\int_0^tS(t-s)\Phi(s)\dd W(s), 
\end{equation*}
where $S$ is the heat semigroup on $L^{\tilde p}(D)$. 
Then, $Z\in L^{\tilde p}(\Om;C([0,T];C(D)))$, and for all $T_1\in(0,T]$,  
\begin{equation*}
    \E[\|Z\|_{C([0,T_1];C(D))}^{\tilde p}]\lesssim_{\tilde p}T^{\frac{\tilde p}{4} - \frac{3}{2}} \E\Big[ \int_0^{T_1} \|\Phi(s,\cdot)\|_{L^{\tilde p}(D)}^{\tilde p}\dd s\Big].
\end{equation*}
\end{theorem}

\subsection{Existence, uniqueness, and regularity for solutions with initial data in $L^q$}\label{sub:wpn}

With the results of the previous subsection at hand, we are ready to begin our study of the controlled SHE for $u\in \A_N$, i.e. 
\begin{align*} 
  X_x^{\eps,u}(t,\xi) = [S(t)x](\xi)+V^{\eps,u}_x(t,\xi)+Y_x^{\eps,u}(t,\xi)+\sqrt{\eps}Z_x^{\eps,u}(t,\xi),
\end{align*}
where $V^{\eps,u}_x(t,\xi) , 
  Y_x^{\eps,u}(t,\xi),$ and $
  Z_x^{\eps,u}(t,\xi)$ are as in \eqref{eq:split Y Z}.  
In the main results of this subsection, we allow for $\eps\in[0,\infty)$, thereby treating the skeleton equation ($\eps=0$), the controlled SHE ($\eps>0$), and the original SHE ($\eps>0,u=0$) simultaneously.

Throughout this subsection, we let 
\begin{align}\label{eq:defE_TF_T}
\begin{alignedat}{2}
E_T &\ceqq C([0,T];L^p(D)), 
&\qquad \|\cdot\|_{E_T} &\ceqq \|\cdot\|_{C([0,T];L^p(D))},\\
F_T &\ceqq C([0,T];C(D)),
&\qquad \|\cdot\|_{F_T} &\ceqq \|\cdot\|_{C([0,T];C(D))}.
\end{alignedat}
\end{align}
Our goal is to prove well-posedness and regularity results for initial data belonging only to $L^q(D)$, allowing $q<p$ and $q<\infty$. In this case, solutions $X_x^{\eps,u}$ cannot be $E_T$- or $F_T$-valued because of the initial condition. 
However, we will establish well-posedness and regularity by studying $X_x^{\eps,u}-S(\cdot)x$, which exhibits greater regularity.
First, in Theorem \ref{th:wpnLp}, we assume the parameter conditions of Theorem \ref{th:Lp} and establish these properties using the space $E_T$. Then, in Theorem \ref{th:wpnC}, we prove an analogous result using the space $F_T$ under the conditions of Theorem \ref{th:C}. 

For our first theorem, the following lemma will be essential to setting up a suitable Picard iteration. 
  
\begin{lemma}\label{lem:reg_ET}
Let the conditions of Theorem \ref{th:Lp} be satisfied. Let 
\[
\tilde{p}\in
\begin{cases}
(4,\frac{2pq}{\lambda p-q}), \, &\text{if } q<\lambda p,\\
(4,\infty), &\text{if } q\geq \lambda p. 
\end{cases}
\]  
Let  $N\in[0,\infty)$, let $u\in L^2(0,T;L^2(D))$ with $\|u\|_{L^2(0,T;L^2(D))}\leq N$, and let $x\in L^q(D)$. 
Define mappings $\Psi_x^b, \Psi_x^{\sigma,u}\col E_T \to E_T$  and $\Psi_x^\sigma\col L^{\tilde p}(\Om;E_T)\to L^{\tilde p}(\Om;E_T)$ by
\begin{align}\label{eq:defs Psi maps}
\begin{split}
    \Psi_x^b(\phi)(t,\xi) &\ceqq\int_0^t \int_D G(t-s,\xi-\eta) b(S(s)x(\eta) + \phi(s,\eta))\dd\eta \dd s\\
    \Psi_x^{\sigma,u}(\phi)(t,\xi) &\ceqq \int_0^t \int_D G(t-s,\xi-\eta) \sigma(S(s)x(\eta) + \phi(s,\eta) )u(s,\eta) \dd\eta \dd s\\
    \Psi_x^\sigma(\phi)(t,\xi) &\ceqq \int_0^t \int_D G(t-s,\xi-\eta) \sigma(S(s)x(\eta) + \phi(s,\eta))W(\dd\eta \dd s).
    \end{split}
\end{align}  

The mappings above are well-defined, and the following Lipschitz bounds hold for all   $\phi,\psi \in E_T$,   $\tilde\phi,\tilde\psi\in L^{\tilde p}(\Om;E_T)$, and $t\in(0,T]$:
\begin{align}\label{eq:PsiLip}
\begin{split}
\|\Psi_x^{b}(\phi)-\Psi_x^{b}(\psi)\|_{E_t}^{\tilde{p}} 
    &\lesssim_{T,\tilde p} \int_0^t \|\phi-\psi\|_{E_s}^{\tilde{p}}\dd s,\\ 
    \|\Psi_x^{\sigma,u}(\phi)-\Psi_x^{\sigma,u}(\psi)\|_{E_t}^{\tilde p} 
    &\lesssim_{N,T,p,\tilde p} \int_0^t \|\phi-\psi\|_{E_s}^{\tilde p} \dd s,\\ 
    \|\Psi_x^{\sigma}(\tilde\phi)-\Psi_x^{\sigma}(\tilde\psi)\|_{L^{\tilde{p}}(\Om;E_t)}^{\tilde{p}}
    &\lesssim_{T,p,\tilde p} \int_0^t \|\tilde\phi-\tilde\psi\|_{L^{\tilde{p}}(\Om;E_s)}^{\tilde{p}}\dd s.  
 \end{split}
 \end{align}   
Furthermore,  $\Psi_x^b$ and $\Psi_x^{\sigma,u}$ with $u\in \A_N$ (see \eqref{eq:AN}) are also well-defined as mappings $L^{\tilde p}(\Om;E_T)\to L^{\tilde p}(\Om;E_T)$, and the first two bounds in \eqref{eq:PsiLip} hold a.s.\ for $\phi,\psi\in L^{\tilde p}(\Om;E_T)$. 
\end{lemma} 
\begin{proof}  
\textbf{Step 1a ($\Psi_x^b$ is well-defined):}
Let $x \in L^1(D) \supset L^q(D)$ and $\phi \in E_T$. By contractiveness of $S$ on $L^p(D)$, \eqref{eq:est}, and the linear growth in \eqref{eq:growth} for $b$, we have  
\begin{align*}
\|\Psi_x^b(\phi)(t,\cdot)\|_{L^p(D)}
&\leq   \int_0^t  \|b(S(s)x+\phi(s))\|_{L^p(D)}\dd s \\
&\lesssim \int_0^t  \|1+|S(s) x| + |\phi(s)|\|_{L^p(D)}\dd s\\
&\lesssim_{T,p,|D|}\int_0^t  1+ s^{-\frac12(1-\frac1p)}\|x\|_{L^1(D)}+\|\phi(s)\|_{L^p(D)} \dd s\\
&\lesssim_{T,p}  1 +   \|x\|_{L^1(D)} +  \|\phi\|_{C([0,t];L^p(D))},
\end{align*}  
using that $-\frac12(1-\frac1p)>-\frac12>-1$. 
Moreover, by the estimates above, we have $b(S(\cdot)x+\phi)\in L^1([0,T];L^{p}(D))$. By classical continuous semigroup theory, it follows that $\Psi_x^b(\phi)=S*b(S(\cdot)x+\phi)\in C([0,T];L^p(D))= E_T$.  
The estimate above gives
\[
\|\Psi_x^b(\phi)\|_{E_T}\lesssim_{T,p,|D|}1 +   \|x\|_{L^1(D)} +  \|\phi\|_{E_T}, 
\]
proving that $\Psi_x^b\col E_T\to E_T$ is well-defined. 

\textbf{Step 1b ($\Psi_x^{\sigma,u}$ is well-defined):} 
Let $x\in L^q(D)$. First, we show that $S(\cdot)x\in L^{\lambda \tilde p}(0,T;L^{\lambda p}(D))$. 
If $q \leq \lambda p$, then the semigroup bound  \eqref{eq:est} gives 
\[
\|S(t)x\|_{L^{\lambda p}(D)}^{\lambda\tilde p} \lesssim_{T,\lambda, p,\tilde p,q} t^{-\frac{1}{2} (\frac{1}{q}-\frac{1}{\lambda p}) \lambda \tilde{p}}\|x\|_{L^q(D)}^{\lambda\tilde p}.
\]
If $q> \lambda p$, then $L^q(D)\subset L^{\lambda {p}}(D)$ by H\"older's inequality, and using   the contractivity of $S$ on $L^q(D)$,  
\[\|S(t) x\|_{L^{\lambda p}(D)}^{\lambda\tilde p} \lesssim_{\lambda,p,\tilde p, q,|D|} 
\|S(t) x\|_{L^{q}(D)}^{\lambda\tilde p}\leq 
\|x\|_{L^{q}(D)}^{\lambda\tilde p}.\]
Combining both cases, we have for all $T_1\in(0,T]$:
\begin{align}\label{eq:Sx Lp}
\int_0^{T_1}\|S(s)x\|_{L^{\lambda p}(D)}^{\lambda \tilde p}\dd s 
&\lesssim_{T,\lambda,p,\tilde p,q,|D|}\int_0^{T_1}s^{-\frac12((\frac1q-\frac{1}{\lambda p})\vee 0)\lambda \tilde p}\|x\|_{L^{q}(D)}^{\lambda \tilde p}\dd s \lesssim_{T,\lambda,p,\tilde p,q} \|x\|_{L^q(D)}^{\lambda\tilde p},
\end{align} 
provided that  
\[
-\frac12(\frac{1}{q}-\frac{1}{\lambda p})\lambda\tilde p >-1 \quad\text{ if } q\leq \lambda p.
\]
Assuming $\tilde p>4$ and $q\leq\lambda p$, the latter holds if and only if 
\[ 
\frac{2\lambda p}{p+2}<q,  \quad \tilde{p}\in(4,\frac{2pq}{ \lambda p-q}),
\] which coincides exactly with our assumptions for the case $q\leq\lambda p$.
Thus indeed,  $S(\cdot)x\in L^{\lambda \tilde p}(0,T;L^{\lambda p}(D))$. 

Let $\phi\in E_T$. To prove that $\Psi_x^{\sigma,u}(\phi)\in E_T$, we can apply Lemma \ref{lem:det conv}  with $p_1=p_2=p$ and $\varphi=\sigma(S(\cdot)x+\phi)$, provided that $\varphi =\sigma(S(\cdot)x+\phi)\in  L^{\tilde p}( [0,T];L^p(D))$. By the growth bound \eqref{eq:growth} for $\sigma$ and by \eqref{eq:Sx Lp}, we have
\begin{align}\label{eq:varphi moment} 
\int_0^T \|\sigma(S(s)x+\phi(s,\cdot))\|_{L^p(D)}^{\tilde p}\dd s
&\lesssim_{\lambda,\tilde p} \int_0^T \|1\|_{L^p(D)}^{\tilde p}+ \|S(s)x\|_{L^{\lambda p}(D)}^{\lambda \tilde p}+ \|\phi(s,\cdot)\|_{L^{\lambda p}(D)}^{\lambda \tilde p}\dd s\notag\\
&\lesssim_{T,\lambda,p,\tilde p,q,|D|} 1+\|x\|_{L^q(D)}^{\lambda\tilde p}+T\|\phi \|_{E_T}^{\lambda \tilde p}\notag\\
&\lesssim_T 1+\|x\|_{L^q(D)}^{\lambda\tilde p}+\|\phi \|_{E_T}^{\tilde p}<\infty,  
\end{align}
where we used that $\lambda\in(0,1]$. 
Thus, Lemma \ref{lem:det conv} indeed gives  
\[
\|\Psi_x^{\sigma,u}(\phi)\|_{E_T}^{\tilde p}\lesssim_{N,T,\lambda,p,\tilde p,q,|D|} 1+\|x\|_{L^q(D)}^{\lambda\tilde p}+\|\phi \|_{E_T}^{\tilde p} <\infty. 
\]

\textbf{Step 1c ($\Psi_x^\sigma$ is well-defined):} 
Let $\tilde\phi\in L^{\tilde p}(\Om;E_T)$. 
We apply Proposition \ref{prop:stoch conv}  with $p_1=p_2=p$ and    $\varphi=\sigma(S(\cdot)x+\tilde\phi)$. Thanks to this result, it suffices to prove that $\varphi\in L^{\tilde p}(\Om;L^{\tilde p}( [0,T];L^p(D)))$. 
Applying \eqref{eq:varphi moment} pointwise in $\om\in\Om$ and taking the expectation yields 
\begin{align}\label{eq:varphi moment 2}
\E[\int_0^T \|\sigma(S(s)x+\tilde\phi(s,\cdot))\|_{L^p(D)}^{\tilde p}\dd s]
&\lesssim_{T,\lambda,p,\tilde p,q,|D|}   1+\|x\|_{L^q(D)}^{\lambda\tilde p}+ \E[\|\tilde\phi \|_{E_T}^{\tilde p}]
<\infty,  
\end{align} 
thus $\Psi_x^\sigma$ is well-defined.   

\textbf{Step 2a ($\Psi_x^b$ is Lipschitz):}
Let $\phi,\psi\in E_T$. 
Since $S$ is contractive  on $L^p(D)$ and $b$ is Lipschitz continuous, we have  for all $0\leq r\leq t\leq T$:
\begin{align*}
\|\Psi_x^b(\phi)(r)-\Psi_x^b(\psi)(r)\|_{L^p(D)}
&\leq \int_0^r \big\| S(r-s) \big( b(S(s)x+\phi(s))-b(S(s)x+\psi(s)) \big) \big\|_{L^p(D)} \dd s \\
&\leq \int_0^r \| b(S(s)x+\phi(s))-b(S(s)x+\psi(s)) \|_{L^p(D)} \dd s \\
&\lesssim \int_0^t \| \phi(s)-\psi(s) \|_{L^p(D)} \dd s\\
&\leq \int_0^t \| \phi-\psi \|_{E_s} \dd s.
\end{align*} 
Taking the supremum over   $r \in [0,t]$, raising both sides to the $\tilde p$-th power, and applying H\"older's inequality (recall that $\tilde p>4 \geq 1$) yields
\begin{align*}
\|\Psi_x^b(\phi)-\Psi_x^b(\psi)\|_{E_t}^{\tilde p}
&\lesssim \Big(\int_0^t \| \phi-\psi \|_{E_s} \dd s\Big)^{\tilde p}\lesssim_{T,\tilde p}\int_0^t \| \phi-\psi \|_{E_s}^{\tilde p} \dd s.
\end{align*}

\textbf{Step 2b ($\Psi_x^{\sigma,u}$ is Lipschitz):}
Lemma \ref{lem:det conv} and the Lipschitz continuity of $\sigma$ imply that for all $t\in[0,T]$ and $\phi,\psi\in E_T$, 
\begin{align*}
    \|\Psi_x^{\sigma,u}(\phi)-\Psi_x^{\sigma,u}(\psi)\|_{E_t}^{\tilde p} &\lesssim_{N,T,p,\tilde p} \int_0^t \|\sigma(S(s)x+\phi(s,\cdot))-\sigma(S(s)x+\psi(s,\cdot))\|_{L^p(D)}^{\tilde p}\dd s\\
    &\lesssim \int_0^t \| \phi  - \psi  \|_{E_s}^{\tilde p}\dd s.
\end{align*}

\textbf{Step 2c ($\Psi_x^{\sigma}$ is Lipschitz):}
Proposition \ref{prop:stoch conv}, the Lipschitz continuity of $\sigma$, and Tonelli's theorem imply that for all $t\in[0,T]$ and $\tilde \phi, \tilde\psi\in L^{\tilde p}(\Om;E_T)$: 
\begin{align*}
    \|\Psi_{\sigma}(\tilde\phi)-\Psi_{\sigma}(\tilde\psi)\|_{L^{\tilde p}(\Om;E_t)}^{\tilde p} &\lesssim_{T,p,\tilde p} 
     \E\Big[\int_0^t \|\sigma(S(s)x+\tilde\phi(s,\cdot))-\sigma(S(s)x+\tilde\psi(s,\cdot))\|_{L^p(D)}^{\tilde p}\dd s\Big]\\
     &\lesssim 
    \int_0^t \|\tilde\phi  -\tilde\psi \|_{L^{\tilde p}(\Om;E_s)}^{\tilde p}\dd s. 
\end{align*} 

\textbf{Step 3 ($\Psi_x^b$ and $\Psi_x^{\sigma,u}$ on $L^{\tilde p}(\Om;E_T)$):} 
Let $\tilde \phi,\tilde\psi\in L^{\tilde p}(\Om;E_T)$ and $u\in\A_N$. Applying the estimates of Steps 1a and 1b pointwise in $\om\in\Om$ gives $\Psi_x^b(\tilde\phi),\Psi_x^{\sigma,u}(\tilde\phi)\in L^{\tilde p}(\Om;E_T)$, and the Lipschitz estimates of \eqref{eq:PsiLip} apply a.s. 
\end{proof}

Next, we establish  existence and uniqueness for the controlled stochastic heat eaquation \eqref{eq:split X}. 

\begin{theorem}\label{th:wpnLp}
Let the conditions in Theorem \ref{th:Lp} be satisfied. Let $\eps,N\in[0,\infty)$,   $u\in \A_N$ (defined in \eqref{eq:AN}), and let $x\in L^q(D)$. 

Then, there exists a mild solution $X_x^{\eps,u}$ to \eqref{eq:split X} in the sense of Definition \ref{def:sol}, which  satisfies 
\begin{align*}
&X_x^{\eps,u}\in L^{\lambda \tilde p}(0,T;L^{\lambda p}(D))+ L^{\tilde{p}}(\Omega;C([0,T];L^p(D)))\subset L^{\tilde p}(\Om;L^{\lambda \tilde p}(0,T;L^{\lambda p}(D)))\\
&\qquad\text{ for all }\quad \tilde{p}\in
\begin{cases}
(4,\frac{2pq}{\lambda p-q}), \, &\text{if } q<\lambda p,\\
(4,\infty), &\text{if } q\geq \lambda p. 
\end{cases}
\end{align*}
 Furthermore, $V^{\eps,u}_x$, $Y^{\eps,u}_x$, $Z^{\eps,u}_x$ (see  \eqref{eq:split Y Z}), and $X^{\eps,u}_x -S(\cdot)x$   all belong to  $L^{\tilde{p}}(\Omega;C([0,T];L^p(D)))$.  
 Uniqueness holds amongst mild solutions belonging to $\{S(\cdot)x+\phi:\phi\in L^{\tilde p}(\Om;C([0,T];L^p(D)))\}$. 
\end{theorem}
\begin{proof} Define the space $E_T$ by \eqref{eq:defE_TF_T}. We use a Picard iteration scheme in the space $L^{\tilde p}(\Om;E_T)$. 
Let $\phi_0\ceqq 0$. Using the mappings from Lemma \ref{lem:reg_ET}, define recursively for $n\in\N$:
\begin{align}\label{eq:phi_n recursion}
\phi_n &\ceqq  \Psi_x^b(\phi_{n-1})+\Psi_x^{\sigma,u}(\phi_{n-1})+\sqrt{\eps} \Psi_x^\sigma(\phi_{n-1}). 
\end{align}  
By the Lipschitz bounds of Lemma \ref{lem:reg_ET}, we have for all $n\in\N$ and $t\in[0,T]$:
\begin{align*}
\|\phi_{n+1}-\phi_n\|_{L^{\tilde p}(\Om;E_t)}^{\tilde p} 
&\lesssim_{\tilde p} 
\|\Psi_x^b(\phi_{n})-\Psi_x^b(\phi_{n-1})\|_{L^{\tilde p}(\Om;E_t)}^{\tilde p}\\
&\qquad+\|\Psi_x^{\sigma,u}(\phi_{n})-\Psi_x^{\sigma,u}(\phi_{n-1})\|_{L^{\tilde p}(\Om;E_t)}^{\tilde p}+
\eps^{\tilde p/2}\|\Psi_x^\sigma(\phi_{n})-\Psi_x^\sigma(\phi_{n-1})\|_{L^{\tilde p}(\Om;E_t)}^{\tilde p} \\
&\lesssim_{N,T,p,\tilde p,\eps} \int_0^t \| \phi_n - \phi_{n-1}\|_{L^{\tilde p}(\Om;E_s)}^{\tilde p}\dd s. 
\end{align*}
By induction, using the preceding estimate, we obtain for all $n\in\N_0$ and $t\in[0,T]$:  
\[\|\phi_{n+1} - \phi_n\|_{L^{\tilde{p}}(\Om;E_t)}^{\tilde p} \leq \|\phi_1 - \phi_0\|_{L^{\tilde{p}}(\Om;E_T)}^{\tilde p} \frac{C^n t^n}{n!}, \]
for a constant $C$ depending only on $N,T,p,\tilde p$ and $\eps$. Consequently,
\[
\sum_{n=0}^{\infty}
\|\phi_{n+1}-\phi_n\|_{L^{\tilde p}(\Om;E_T)}
\leq
\|\phi_1-\phi_0\|_{L^{\tilde p}(\Om;E_T)}
\sum_{n=0}^{\infty}
\Big(\frac{C^nT^n}{n!}\Big)^{1/\tilde p}
<\infty.
\]
Since $\|\phi_{m}-\phi_k\|_{L^{\tilde p}(\Om;E_T)}\leq \sum_{n=k}^\infty \|\phi_{n+1}-\phi_n\|_{L^{\tilde p}(\Om;E_T)}$, it follows that   $(\phi_n)_n$ is
Cauchy in $L^{\tilde p}(\Om;E_T)$, and therefore, it has a limit $\phi$ in $L^{\tilde p}(\Om;E_T)$. 
Moreover, 
Tonelli's theorem yields
\begin{align*}
\E\Big[
\sum_{n=0}^{\infty}
\|\phi_{n+1}-\phi_n\|_{E_T}
\Big]=
\sum_{n=0}^{\infty}
\E\Big[
\|\phi_{n+1}-\phi_n\|_{E_T}
\Big] \leq
\sum_{n=0}^{\infty}
\|\phi_{n+1}-\phi_n\|_{L^{\tilde p}(\Om;E_T)}
<\infty.
\end{align*}
Hence, $\sum_{n=0}^{\infty}
\|\phi_{n+1}-\phi_n\|_{E_T}<\infty$ a.s.  
By a telescoping sum, it follows that 
$(\phi_n)_n$ is
Cauchy in $E_T$ a.s., hence $(\phi_n)_n$ converges also a.s.\ to $\phi$ in $E_T$ (by uniqueness of limits in probability).

Define $X\ceqq S(\cdot)x+ \phi$ and note that $X\in L^{\lambda \tilde p}(0,T;L^{\lambda p}(D))+ L^{\tilde p}(\Om;E_T)$ by \eqref{eq:Sx Lp}. We prove that $X$ is a mild solution to \eqref{eq:split X}.  
Since $S(\cdot)x$ is deterministic and Borel measurable, it is progressively measurable.  Moreover, since $\phi_n\to\phi$ a.s.\ in $E_T$ and each $\phi_n$ is adapted, completeness of the filtration implies that $\phi$ has an indistinguishable adapted continuous version, which is progressively measurable. Hence $X$ is indistinguishable from a progressively measurable process. 
Next, we verify that $X$ satisfies the mild solution formula. 
By \eqref{eq:PsiLip} and Tonelli's theorem,  
\begin{align*}
&\|\Psi_x^{b}(\phi_n) -\Psi_x^{b}(\phi)\|_{L^{\tilde p}(\Om;E_T)}+\|\Psi_x^{\sigma,u}(\phi_n)-\Psi_x^{\sigma,u}(\phi)\|_{L^{\tilde p}(\Om;E_T)}+ \|\Psi_x^{\sigma}(\phi_n)-\Psi_x^{\sigma}(\phi)\|_{L^{\tilde p}(\Om;E_T)} \\
&\lesssim_{N,T,p,\tilde p}\textstyle{\int_0^T} \|\phi_n-\phi\|_{L^{\tilde p}(\Om;E_s)}\dd s \leq T\|\phi_n-\phi\|_{L^{\tilde p}(\Om;E_T)}\to 0\qquad   \text{as }n\to\infty.
\end{align*}
In particular, we conclude that 
the following convergences hold in $L^p(D)$, a.s.\ for all $t\in [0,T]$:
\[
\lim_{n\to \infty}\Psi_x^{b}(\phi_n)(t)=\Psi_x^{b}(\phi)(t),\quad
\lim_{n\to \infty}\Psi_x^{\sigma,u}(\phi_n)(t)=\Psi_x^{\sigma,u}(\phi)(t),\quad
\lim_{n\to \infty}\Psi_x^{\sigma}(\phi_n)(t)=\Psi_x^{\sigma}(\phi)(t).
\]
Consequently, we have a.s.  for all $t\in [0,T]$, a.e.\ on $D$: 
\begin{align*}
X(t)=S(t)x+\phi &=S(t)x+\lim_{n\to\infty} {\phi}_n(t)\\
&=S(t)x+\lim_{n\to \infty}\big(\Psi_x^{b}(\phi_n)(t)+\Psi_x^{\sigma,u}(\phi_n)(t)+\sqrt{\eps}\Psi_x^{\sigma}(\phi_n)(t)\big)\\
&=S(t)x+\Psi_x^{b}(\phi)(t)+\Psi_x^{\sigma,u}(\phi)(t)+\sqrt{\eps}\Psi_x^{\sigma}(\phi)(t)\\
&= S(t)x+ \Psi_x^{b}(-S(\cdot)x+X)(t)+\Psi_x^{\sigma,u}(-S(\cdot)x+X)(t)+\sqrt{\eps}\Psi_x^{\sigma}(-S(\cdot)x+X)(t).
\end{align*}
Recalling the definition of $\Psi_x^{b}, \Psi_x^{b},$ and $\Psi_x^{b}$, this proves that $X_x^{\eps,u}\ceqq X$ is a mild solution. 

For the regularity claims, note that by   construction, $X_x^{\eps,u}-S(\cdot)x=\phi\in L^{\tilde p}(\Om;E_T)$. Moreover, 
\[
V_x^{\eps,u}=\Psi_x^{b}(X_x^{\eps,u}-S(\cdot)x), \quad 
Y_x^{\eps,u}=\Psi_x^{\sigma,u}(X_x^{\eps,u}-S(\cdot)x), \quad
Z_x^{\eps,u}=\Psi_x^{\sigma}(X_x^{\eps,u} -S(\cdot)x). 
\]
Since  the $\Psi$ mappings were proved to map   $L^{\tilde p}(\Om;E_T)\to L^{\tilde p}(\Om;E_T)$, we conclude that $V_x^{\eps,u}, Y_x^{\eps,u}$ and $Z_x^{\eps,u}$ belong to $L^{\tilde p}(\Om;E_T)$. 

Finally, we prove uniqueness within $\Gamma\ceqq \{S(\cdot)x+\phi:\phi\in L^{\tilde p}(\Om;E_T)\}$.  If $X_x^{\eps,u,1}, X_x^{\eps,u,2}$ are mild solutions and belong to $\Gamma$, then for $j=1,2$: $X_x^{\eps,u,j}-S(\cdot)x \in L^{\tilde p}(\Om;E_T)$ and 
\[
X_x^{\eps,u,j}-S(\cdot)x=\Psi_x^b(X_x^{\eps,u,j}-S(\cdot)x)+\Psi_x^{\sigma,u}(X_x^{\eps,u,j}-S(\cdot)x)+\sqrt{\eps}\Psi_x^\sigma(X_x^{\eps,u,j}-S(\cdot)x).
\]
Hence, \eqref{eq:PsiLip} applied with $\phi=X_x^{\eps,u,1}-S(\cdot)x$ and $\psi=X_x^{\eps,u,2}-S(\cdot)x$ implies that for all $t\in[0,T]$:  
$$
\|X_x^{\eps,u,1}-{X}_x^{\eps,u,2}\|_{L^{\tilde p}(\Om;E_t)}\lesssim_{N,T,p,\tilde p,\eps}\int_0^t\|X_x^{\eps,u,1}- {X}_x^{\eps,u,2}\|_{L^{\tilde p}(\Om;E_s)}\dd s,
$$ 
and Gr\"onwall's inequality yields $X_x^{\eps,u,1}-{X}_x^{\eps,u,2}=0$ a.s., proving uniqueness. 
\end{proof}

Next, we work toward our second well-posedness and regularity result, Theorem \ref{th:wpnC}, which treats   initial data in $L^q(D)$ under the parameter conditions of Theorem \ref{th:C}. 
As a preliminary step, we derive a suitable analog of Lemma \ref{lem:reg_ET} using the space $F_T$. 

\begin{lemma}\label{lem:regC}
Let the conditions of Theorem \ref{th:C} hold and let $\tilde{p}\in (6,\frac{3q}{\lambda})$.     
Let  $N\in[0,\infty)$, let $u\in L^2(0,T;L^2(D))$ with $\|u\|_{L^2(0,T;L^2(D))}\leq N$, and let $x\in L^q(D)$. 
Define the space $F_T$ as in \eqref{eq:defE_TF_T}, 
and define $\Psi_x^b, \Psi_x^{\sigma,u}\col F_T \to F_T$  and $\Psi_x^\sigma\col L^{\tilde p}(\Om;F_T)\to L^{\tilde p}(\Om;F_T)$ by  the formulas in \eqref{eq:defs Psi maps}. 

Then, these mappings are well-defined and the following Lipschitz bounds hold for all   $\phi,\psi \in F_T$,   $\tilde\phi,\tilde\psi\in L^{\tilde p}(\Om;F_T)$, and $t\in(0,T]$:
\begin{align}\label{eq:PsiLip C}
\begin{split}
\|\Psi_x^{b}(\phi)-\Psi_x^{b}(\psi)\|_{F_t}^{\tilde{p}} 
    &\lesssim_{T,\tilde p} \int_0^t \|\phi-\psi\|_{F_s}^{\tilde{p}}\dd s,\\ 
    \|\Psi_x^{\sigma,u}(\phi)-\Psi_x^{\sigma,u}(\psi)\|_{F_t}^{\tilde p} 
    &\lesssim_{N,T,\tilde p} \int_0^t \|\phi-\psi\|_{F_s}^{\tilde p} \dd s,\\ 
    \|\Psi_x^{\sigma}(\tilde\phi)-\Psi_x^{\sigma}(\tilde\psi)\|_{L^{\tilde{p}}(\Om;F_t)}^{\tilde{p}}
    &\lesssim_{T,\tilde p} \int_0^t \|\tilde\phi-\tilde\psi\|_{L^{\tilde{p}}(\Om;F_s)}^{\tilde{p}}\dd s.  
 \end{split}
 \end{align}  
 Furthermore,  $\Psi_x^b$ and $\Psi_x^{\sigma,u}$ with $u\in \A_N$ (see \eqref{eq:AN}) are also well-defined as mappings $L^{\tilde p}(\Om;F_T)\to L^{\tilde p}(\Om;F_T)$, and the first two bounds in \eqref{eq:PsiLip C} hold a.s.\ for $\phi,\psi\in L^{\tilde p}(\Om;F_T)$. 
\end{lemma}

\begin{proof}
    The proof of Lemma \ref{lem:reg_ET} can be copied with the small   adjustments indicated below.  
Let $\tilde p,N,u,$ and $x$ be as in the statement. 

\textbf{Step 1a}: 
We have by the linear growth of $b$ and \eqref{eq:est}: 
\[ 
\|b(S(s)x+\phi(s))\|_{L^\infty(D)}
 \lesssim
1+\|S(s)x\|_{L^\infty(D)}
+\|\phi(s)\|_{C(D)} \lesssim_T
1+s^{-1/2}\|x\|_{L^1(D)}
+\|\phi(s)\|_{C(D)}.
\]  
Consequently, $b(S(\cdot)x+\phi) \in L^1(0,T;L^\infty(D))$ and since $S$ is a contractive semigroup on $L^\infty(D)$, 
\[ 
\|\Psi_x^b(\phi)(t)\|_{L^\infty(D)}
 \leq
\int_0^t\|b(S(s)x+\phi(s))\|_{L^\infty(D)} \dd s
 \lesssim_T
1+ \|x\|_{L^1(D)}
+ \|\phi\|_{F_T}. 
\]
For continuity  of $\Psi_x^b(\phi)$ in space and time, note that  $b(S(\cdot)x + \phi)\in L^1(0,T;C(D))$ by the smoothing properties of the heat semigroup and Lipschitz continuity of $b$.  Moreover, $S$ is strongly continuous on $C(D)$, so standard semigroup theory yields $S*b(S(\cdot)x + \phi)\in C([0,T];C(D))$.

\textbf{Step 1b}: 
We  derive an $L^{\lambda \tilde p}$ bound for the term $S(\cdot)x$. 
When $q \leq \lambda \tilde{p}$, \eqref{eq:est} guarantees that
\[
\|S(t)x\|_{L^{\lambda \tilde{p}}(D)} \lesssim_{T,\lambda, \tilde p,q} t^{-\frac{1}{2} (\frac{1}{q} - \frac{1}{\lambda \tilde p})}\|x\|_{L^q(D)}.
\]
When $q> \lambda \tilde{p}$, we have $L^q(D)\subset L^{\lambda \tilde{p}}(D)$  and  
\[\|S(t) x\|_{L^{\lambda \tilde{p}}(D)} \lesssim_{\lambda,\tilde p, q,|D|} 
\|S(t) x\|_{L^{q}(D)}\leq 
\|x\|_{L^{q}(D)}.\]
Combining both cases, we obtain  
\begin{align}\label{eq:Sx bnd C case}
  \|S(t)x\|_{L^{\lambda \tilde p}(D)}  
  \lesssim_{T,\lambda, \tilde p,q,|D|} t^{-\frac12((\frac1q-\frac{1}{\lambda \tilde p})\vee 0)}\|x\|_{L^q(D)}. 
\end{align}
If $q\leq \lambda \tilde p$, we have $S(\cdot)x\in L^{\lambda \tilde p}([0,T]\times D)$ if  $-\frac12(\frac1q-\frac{1}{\lambda \tilde p})\lambda \tilde p>-1$, or equivalently, $\tilde p<\frac{3q}{\lambda}.$ Our assumption that $\tilde p\in (6,\tfrac{3q}{\lambda})$ thus guarantees that $S(\cdot)x\in L^{\lambda \tilde p}([0,T]\times D)$.

Consequently, for $\phi\in F_T$, we have $\varphi\ceqq \sigma(S(\cdot)x+\phi)\in L^{\tilde p}([0,T]\times D)$  by \eqref{eq:varphi moment} applied with $p=\tilde p$ and noting that $F_T\into E_T$. Hence, for well-definedness of $\Psi_x^{\sigma,u}\col F_T\to F_T$,  we can apply Lemma \ref{lem:det conv} with $p_1\ceqq \tilde p$  and $p_2\ceqq\infty$. Note that for these $p_1$ and $p_2$, \eqref{eq:cond tilde p} holds if and only if $\tilde p\in(6,\infty)$, which we assumed. 

\textbf{Step 1c}: 
For well-definedness of $\Psi_x^{\sigma}\col L^{\tilde p}(\Om;F_T)\to L^{\tilde p}(\Om;F_T)$, we apply Theorem \ref{th:stoch conv C} instead of Proposition \ref{prop:stoch conv}.  Note that for $\phi\in L^{\tilde p}(\Om;F_T)$, we have $\varphi\ceqq \sigma(S(\cdot)x+\phi)\in L^{\tilde p}(\Om;L^{\tilde p}([0,T]\times D))$ by \eqref{eq:varphi moment 2} applied with $p=\tilde p$ and since $F_T\into E_T$. 

\textbf{Steps 2a--2c}: In the proofs of these steps, we can replace every $E_t$ by $F_t$ and $L^p(D)$ by $L^\infty(D)$, applying Lemma \ref{lem:det conv}   with $p_1=p_2=\infty$ in Step 2b and applying  Theorem \ref{th:stoch conv C} in Step 2c. 
 
\textbf{Step 3}: In the proof we can replace $E_T$ by $F_T$ and $L^p(D)$, $L^{\lambda p}(D)$ by $L^{\tilde p}(D)$, $L^{\lambda\tilde p}(D)$.  
\end{proof}

\begin{theorem}\label{th:wpnC}
Let the conditions in Theorem \ref{th:C} be satisfied. Let $\eps,N\in[0,\infty)$, $u\in \A_N$ (defined in \eqref{eq:AN}), and let  $x\in L^q(D)$. 

Then, there exists a mild solution $X_x^{\eps,u}$ to \eqref{eq:split X} in the sense of Definition \ref{def:sol}, which  satisfies 
\begin{align*}
&X_x^{\eps,u}\in L^{\lambda \tilde p}(0,T;L^{\lambda \tilde p}(D))+ L^{\tilde{p}}(\Omega;C([0,T];C(D)))\subset L^{\tilde p}(\Om;L^{\lambda \tilde p}(0,T;L^{\lambda \tilde p}(D)))\\
&\qquad\text{ for all }\quad \tilde{p}\in (6,\frac{3q}{\lambda}). 
\end{align*}
 Furthermore, $V^{\eps,u}_x$, $Y^{\eps,u}_x$, $Z^{\eps,u}_x$ (see  \eqref{eq:split Y Z}), and $X^{\eps,u}_x -S(\cdot)x$   all belong to  $L^{\tilde{p}}(\Omega;C([0,T];C(D)))$.  
 Uniqueness holds amongst mild solutions belonging to $\{S(\cdot)x+\phi:\phi\in L^{\tilde p}(\Om;C([0,T];C(D)))\}$. 
\end{theorem} 
\begin{proof}
We can copy the Picard iteration scheme from the proof of Theorem \ref{th:wpnLp} with $E_T$ replaced by $F_T$. Since Lemma \ref{lem:regC} fully replaces Lemma  \ref{lem:reg_ET}, all arguments remain valid and yield existence and uniqueness of mild solutions belonging to $\{S(\cdot)x+\phi:\phi\in L^{\tilde p}(\Om;C([0,T];C(D)))\}$. Moreover, the regularity claims can be proved analogously, using the splitting    
$X_x^{\eps,u}=S(\cdot)x+V_x^{\eps,u}+Y_x^{\eps,u}+\sqrt{\eps}Z_x^{\eps,u}$ with $X_x^{\eps,u}-S(\cdot)x\in F_T$. The latter then implies that $V_x^{\eps,u}=\Psi_x^b(X_x^{\eps,u}-S(\cdot)x)\in F_T$, $Y_x^{\eps,u}=\Psi_{\sigma,u}(X_x^{\eps,u}-S(\cdot)x)\in F_T$, and $Z_x^{\eps,u}=\Psi_{\sigma}(X_x^{\eps,u}-S(\cdot)x)\in  F_T$. 
\end{proof}

As a consequence of Theorems \ref{th:wpnLp} and \ref{th:wpnC}, we obtain the following results for classical initial data and solution spaces.  For the uncontrolled case ($u=0$), corresponding results are classical, see e.g.\  \cite{daprato}. For the controlled problem,  part \textit{(2)} below   also follows from \cite[Th.\ 7.1]{Salins21reacdiff}, noting that the boundary conditions do not affect that result.

\begin{corollary}[Well-posedness, classical cases ($q=p$)]\label{cor:wpn classical} 
Suppose that Assumption \ref{asm} holds. Let $\eps,N \in[ 0,\infty)$  and $u\in \A_N$.  
\begin{enumerate}
\item  If $x\in L^p(D)$ and $\tilde p\in(4,\infty)$, then there exists a unique mild solution $X_x^{\eps,u}$ to \eqref{eq:split X} that belongs to $L^{\tilde p}(\Om;C([0,T]; L^p(D)))$. 
\item If $x\in C(D)$ and $\tilde p\in(6,\infty)$, then there exists a unique mild solution $X_x^{\eps,u}$ to \eqref{eq:split X} that belongs to $L^{\tilde p}(\Om;C([0,T];C(D)))$. 
\end{enumerate}
\end{corollary}

\begin{proof}
\textit{(1)}: Let $x\in L^p(D)$ and $\tilde p\in(4,\infty)$. The parameter constraint of Theorem \ref{th:Lp} holds for any $\lambda\in(0,1]$ if $q=p$.  Theorem \ref{th:wpnLp} applied with $q\ceqq p\geq\lambda p$ yields the existence of a mild solution $X_x^{\eps,u}$  with $X_x^{\eps,u}-S(\cdot)x\in L^{\tilde p}(\Om;C([0,T]; L^p(D)))$. Since $S$ is a strongly continuous semigroup on $L^p(D)$, we also have $S(\cdot)x\in C([0,T];L^p(D))$, thus $X_x^{\eps,u} \in L^{\tilde p}(\Om;C([0,T]; L^p(D)))$ and uniqueness follows from the  uniqueness  established in Theorem \ref{th:wpnLp}. 

\textit{(2)}: Let $x\in C(D)$ and $\tilde p\in(6,\infty)$. Fix a sufficiently large $q\in[1,\infty)$ such that $\tilde p<3q/\lambda$ and $2\lambda<q$.  
By Theorem \ref{th:wpnC}  we obtain a mild solution $X_x^{\eps,u}$ with $X_x^{\eps,u}-S(\cdot)x\in L^{\tilde p}(\Om;C([0,T];C(D)))$. Moreover, $S(\cdot)x\in C([0,T];C(D))$, thus $X_x^{\eps,u}$ has the stated regularity and uniqueness follows from the uniqueness established in Theorem \ref{th:wpnC}. 
\end{proof}

\section{ULDP on $C([0,T];\Cper)$ over $L^q$-bounded subsets of $\Cper$}\label{sec:C}

In this section we prove Theorem \ref{th:C}, using the sufficient condition stated in Theorem \ref{th:uldp suff}(1). 
To verify that condition, we will establish uniform estimates for $\|X_x^{\eps,u}-X_x^{0,u}\|_{L^{\tilde p}(\Om;C([0,T];\Cper))}$, for a suitable $\tilde p>6$. Note that we can write
\[
X_x^{\eps,u}-X_x^{0,u}=V_x^{\eps,u}-V_x^{0,u}+Y_x^{\eps,u}-Y_x^{0,u}+\sqrt{\eps}Z_x^{\eps,u}. 
\] 
As we will show in the proof of Theorem \ref{th:C}, the differences $V_x^{\eps,u}-V_x^{0,u}$ and $Y_x^{\eps,u}-Y_x^{0,u}$ can be estimated suitably in terms of $X_x^{\eps,u}-X_x^{0,u}$, and will disappear after an application of  Gr\"onwall's inequality, thus it does not lead to any $x$-dependence. In contrast, $Z_x^{\eps,u}(t,\xi)$ does depend on $x$, and we need to derive   estimates  for this term that are uniform with respect to $L^q$-bounded initial data $x$. 

Let us outline how we will deal with the term  $Z_x^{\eps,u}(t,\xi)$. 
Using Theorem \ref{th:stoch conv C}, it will be shown that for a suitable $\tilde p$, we have
\begin{align*}
\E[\|\eps^{1/2}Z_x^{\eps,u} \|_{C([0,T];C(D))}^{\tilde p}]
&\lesssim_{T, \tilde p,\lambda}\eps^{\tilde p/2}\E\Big[\int_0^T\int_D 1+ |X_x^{\eps,u}-X_x^{0,u}|^{\tilde p}+|X_x^{0,u}|^{\lambda \tilde p}\dd\xi\dd t\Big].  
\end{align*} 
Since we study the limit $\eps\downarrow 0$, we can assume that $\eps\leq \eps_0$, and the second term in the sum above will be absorbed through Gr\"onwall's inequality. 
Due to the prefactor $\eps^{\tilde p/2}$ above, it then remains to prove that  
\begin{equation*} 
  \sup_{u\in\A_N}\sup_{x\in B_R^q\cap C(D)} \E\Big[\int_0^T\int_D  |X_x^{0,u}(t,\xi)|^{\lambda \tilde p}\dd\xi\dd t\Big]<\infty.
\end{equation*} 
In Subsection \ref{sub:I2 C}, we will deal with $X_x^{0,u}$, and establish an even stronger, almost sure bound, which is uniform over $x\in B^q_R$. In Subsection \ref{sub:concl C}, we provide the proof of Theorem \ref{th:C}.

\subsection{Uniform $L^{\lambda\tilde p}$-bound for $X_x^{0,u}$}\label{sub:I2 C}

In the next lemma, we prove an almost sure, uniform  bound for solutions $X_x^{0,u}$ to equation \eqref{eq:split X} with $\eps=0$. 
This bound is uniform with respect to $L^2$-bounded (forcing) functions $u$, and uniform with respect to $L^q$-bounded initial data $x$.

\begin{lemma} 
\label{lem:conv Linfty} Let Assumption \ref{asm} hold. Let $N>0$ and $q\in[1,\infty)$. 
  Suppose that $\lambda\in [0,\frac{q}{2})\cap (0,1]$ and $\tilde p\in (6,\frac{3q}{\lambda})$. 
  Then, there exists a constant $M>0$ such that for all $u\in\A_N$ and  $x\in L^q(D)$: 
  \begin{equation}\label{eq:toshow}
\int_0^T\int_D  |X_x^{0,u}(t,\xi)|^{\lambda \tilde p}\dd\xi\dd t \leq M(1+\|x\|_{L^q(D)}^{\lambda \tilde p}) \quad\text{a.s.}
  \end{equation}
  The value of $M$ depends only on $T,\lambda,\tilde p,q,|D|,$ and $N$.
\end{lemma}

\begin{proof}[Proof of Lemma \ref{lem:conv Linfty}]
The condition $\lambda <{q}/{2}$ is there only to  ensure that the range for $\tilde p$ is non-empty. 
Note that $X_x^{0,u}(\om)=X_x^{0,u(\om)}$ for a.e.\ $\om\in\Om$, i.e.\ stochasticity only enters into $X_x^{0,u}$ through $u(\om)$. Therefore, recalling the definition of $\A_N$, it suffices to prove that for all $u\in L^2(0,T;L^2(D))$ with $\|u\|_{L^2(0,T;L^2(D))}\leq N$, the bound in \eqref{eq:toshow} holds.  

Let $u\in L^2(0,T;L^2(D))$  with $\|u\|_{L^2(0,T;L^2(D))}\leq N$. We will apply Gr\"onwall's inequality to $t\mapsto  \|X_x^{0,u}\|_{L^{\lambda \tilde p}([0,t]\times D)}$.   
Using the notation of \eqref{eq:split Y Z}, we have
\[
  X_x^{0,u}(t,\xi)= 
  [S(t)x](\xi)+V_x^{0,u}(t,\xi)+Y_x^{0,u}(t,\xi). 
\]
By \eqref{eq:Sx bnd C case}, we have for any $\tilde p\in (6,\frac{3q}{\lambda})$:  
\begin{equation}\label{eq:Sx est}
\|S(\cdot)x\|_{L^{\lambda \tilde p}([0,T]\times D)}\lesssim_{T,\lambda,\tilde p,q,|D|} \|x\|_{L^q(D)}.
\end{equation} 
For $V_x^{0,u}$,  assuming for the moment that $\lambda\tilde p\geq 1$, we have 
by Minkowki's inequality, H\"older's inequality and the linear growth of $b$:
\begin{align*}
\|V_x^{0,u}(t,\cdot)\|_{L^{\lambda\tilde p}(D)}^{\lambda\tilde p}
&\leq \Big(\int_0^t\|S(t-s)b(X_x^{0,u}(s,\cdot))\|_{L^{\lambda\tilde p}(D)}\dd s\Big)^{\lambda\tilde p}\\
&\lesssim_{T,\lambda,\tilde p} \int_0^t\|S(t-s)b(X_x^{0,u}(s,\cdot))\|_{L^{\lambda\tilde p}(D)}^{\lambda\tilde p}\dd s \\
&\leq \int_0^t\|b(X_x^{0,u}(s,\cdot))\|_{L^{\lambda\tilde p}(D)}^{\lambda\tilde p}\dd s\\
&\lesssim_{T,\lambda,\tilde p,|D|} 1+ \|X_x^{0,u}\|_{L^{\lambda\tilde p}([0,t]\times D)}^{\lambda\tilde p}.
\end{align*}
For the convolution term $Y_x^{0,u}$, we apply Lemma \ref{lem:det conv}   with $p_1=p_2= \tilde p>6$ and $\varphi\ceqq \sigma(X_x^{0,u})$, and use the growth bound \eqref{eq:growth} for $\sigma$ and the fact that $\lambda \in (0,1]$, giving  for all $t\in[0,T]$:  
\begin{align*}
  \|Y_x^{0,u}(t,\cdot)\|_{L^{\lambda\tilde p}(D)}^{\tilde p}
  \lesssim_{\lambda,\tilde p,|D|}\|Y_x^{0,u}(t,\cdot)\|_{L^{\tilde p}(D)}^{\tilde p} 
  &\lesssim_{\tilde p,N,T} \int_0^t \|\sigma(X_x^{0,u}(s,\cdot))\|_{L^{\tilde p}(D)}^{\tilde p}\dd s \\
  &\lesssim_{\tilde p,|D|} \int_0^t \big(1+\|X_x^{0,u}(s,\cdot)\|_{L^{\lambda \tilde p}(D)}^{\lambda \tilde p}\big)\dd s\\ 
  &\lesssim_T 1 + \|X_x^{0,u}\|_{L^{\lambda \tilde p}([0,t]\times D)}^{\tilde p}.
\end{align*}
Raising this estimate to the power $\lambda$ gives
\begin{align*}
  \|Y_x^{0,u}(t,\cdot)\|_{L^{\lambda\tilde p}(D)}^{\lambda\tilde p} \lesssim_{\lambda,\tilde p,|D|, N,T}   1 + \|X_x^{0,u}\|_{L^{\lambda\tilde p}([0,t]\times D)}^{\lambda\tilde p}.
\end{align*}
For all $t\in[0,T]$, we conclude that 
\begin{align*}
 \|X_x^{0,u}(t,\cdot)\|_{L^{\lambda \tilde p}(D)}^{\lambda \tilde p}
 &\lesssim_{\lambda, \tilde p} \|S(t)x\|_{L^{\lambda \tilde p}(D)}^{\lambda \tilde p}+\|V_x^{0,u}(t,\cdot)\|_{L^{\lambda \tilde p}(D)}^{\lambda \tilde p}+\|Y_x^{0,u}(t,\cdot)\|_{L^{\lambda \tilde p}(D)}^{\lambda \tilde p} \\
 &\lesssim_{\lambda,\tilde p,|D|, N,T} \|S(t)x\|_{L^{\lambda \tilde p}(D)}^{\lambda \tilde p}+ 1 + \|X_x^{0,u} \|_{L^{\lambda \tilde p}([0,t]\times D)}^{\lambda \tilde p}. 
\end{align*}
Hence, using also \eqref{eq:Sx est} and integrating over $t$,  we find  for all $0\leq  T_1\leq T$:
\begin{align*}
\|X_x^{0,u} \|_{L^{\lambda \tilde p}([0,T_1]\times D)}^{\lambda \tilde p}&= \int_0^{T_1}\|X_x^{0,u}(t,\cdot)\|_{L^{\lambda \tilde p}(D)}^{\lambda \tilde p}\dd t\\
 &\lesssim_{T,\lambda,\tilde p,q,|D|, N} \|x\|_{L^q(D)}^{\lambda \tilde p}+ 1 +\int_0^{T_1} \|X_x^{0,u} \|_{L^{\lambda \tilde p}([0,t]\times D)}^{\lambda \tilde p}\dd t.  
\end{align*} 
Gr\"onwall's inequality yields
\[
\|X_x^{0,u} \|_{L^{\lambda \tilde p}([0,T]\times D)}^{\lambda \tilde p}\leq \hat C(\|x\|_{L^q(D)}^{\lambda \tilde p}+ 1)\exp(\hat C),
\]
where $\hat C$ is a constant depending only on $T,\lambda,\tilde p,q,|D|,$ and $N$. Setting  $M\ceqq\hat C \exp(\hat C)$, this concludes the proof if $\lambda\tilde p\geq 1$.  

If $\lambda\tilde p<1$, then we put $\tilde p_*\ceqq 1/\lambda$. Note that $\tilde p_*\lambda=1$ and $\tilde p_*\in(6,\frac{3q}{\lambda})$ since $\tilde p_*>\tilde p>6$ and $q\geq 1$. Thus the proof applies with $\tilde p=\tilde p_*$ and gives for all $u\in\A_N$: 
\[
\int_0^T\int_D  |X_x^{0,u}(t,\xi)| \dd\xi\dd t \leq M(1+\|x\|_{L^q(D)})\quad \text{ a.s.}
\]
Combined with Jensen's inequality (for the concave function $|\cdot|^{\lambda \tilde p}$), we obtain a.s.
\begin{align*}
\int_0^T\int_D  |X_x^{0,u}(t,\xi)|^{\lambda \tilde p}\dd\xi\dd t\lesssim_{|D|,T,\lambda,\tilde p}
    \Big(\int_0^T\int_D  |X_x^{0,u}(t,\xi)| \dd\xi\dd t\Big)^{\lambda\tilde p} 
    &\leq (M(1+\|x\|_{L^q(D)}))^{\lambda\tilde p}\\
    &\leq M^{\lambda\tilde p}(1+\|x\|_{L^q(D)}^{\lambda \tilde p}).
\end{align*} 
Redefining $M$, we again obtain \eqref{eq:toshow}.
\end{proof}

\subsection{Conclusion of the ULDP}\label{sub:concl C}

Combining Lemma \ref{lem:conv Linfty} with   results from Section \ref{sec:wpn estimates}, we can now prove the ULDP of Theorem \ref{th:C}. 

\begin{proof}[Proof of Theorem \ref{th:C}] 
Fix any $\tilde p\in(6,\frac{3q}{\lambda})$ and $N,R\in[0,\infty)$. For all $u\in\A_N$ and $x\in L^q(D)$,  Theorem \ref{th:wpnC} provides a unique mild solution  $X_x^{\eps,u}\in L^{\tilde p}(\Om; L^{\lambda\tilde p}(0,T;L^{\lambda\tilde p}(D)))$ to \eqref{eq:split X}. Using this regularity and the growth condition in \eqref{eq:growth}, we have $\sigma(X_x^{\eps,u})\in  
L^{\tilde p}(\Om;L^{\tilde p}(0,T;L^{\tilde p}(D)))$. 
Therefore, we may apply Theorem \ref{th:stoch conv C} with $\Phi\ceqq \sigma(X_x^{\eps,u})$, giving together with \eqref{eq:growth}:  
\begin{align*}
\E[ \|\eps^{1/2}Z_x^{\eps,u}&\|_{C([0,T];\Cper)}^{\tilde p}] 
\lesssim_{\tilde p,T}\eps^{\tilde p/2}\E\Big[\int_0^T\int_D 1+|X_x^{\eps,u}|^{\lambda \tilde p}\dd\xi\dd t\Big]\\
&\lesssim_{\lambda, \tilde p}\eps^{\tilde p/2}\E\Big[\int_0^T\int_D 1+ |X_x^{\eps,u}-X_x^{0,u}|^{\tilde p}+|X_x^{0,u}|^{\lambda \tilde p}\dd\xi\dd t\Big] \\
&\lesssim_{T,|D|}\eps^{\tilde p/2}\bigg(1+\int_0^T \E\big[\|X_x^{\eps,u}-X_x^{0,u}\|_{C([0,t];\Cper)}^{\tilde p}\big]\dd t+\E\Big[\int_0^T\int_D  |X_x^{0,u}|^{\lambda \tilde p}\dd\xi\dd t\Big]\bigg). 
\end{align*}
In the third line we used that $(|z_1|+|z_2|)^{\lambda \tilde p}\leq 2^{\lambda \tilde p-1}( |z_1|^{\lambda \tilde p}+|z_2|^{\lambda \tilde p})$, and we used that $|z|^\lambda\leq 1+|z|$ since $\lambda\in(0,1]$. 
For the last term in the upper bound above, we have by Lemma \ref{lem:conv Linfty} (see \eqref{eq:BqR}): 
\[
M_*\ceqq\sup_{u\in\A_N}\sup_{x\in B_R^q} \E\Big[\int_0^T\int_D  |X_x^{0,u}(t,\xi)|^{\lambda\tilde p}\dd\xi\dd t\Big]\lesssim_{T,\lambda,\tilde p,q,|D|, N,R}1<\infty.
\]
Furthermore, by \eqref{eq:PsiLip C} of Lemma  \ref{lem:regC} ($V_x^{\eps,u}=\Psi_x^b(X_x^{\eps,u}-S(\cdot)x)$, $Y_x^{\eps,u}=\Psi_x^{\sigma,u}(X_x^{\eps,u}-S(\cdot)x)$ for $\eps\geq 0$) and the Tonelli's theorem, we have for all $u\in\A_N$ and $x\in L^q(D)$: 
\begin{align*}
\E[&\|V_x^{\eps,u}-V_x^{0,u}\|_{C([0,T_1];\Cper)}^{\tilde p}]+
\E[\|Y_x^{\eps,u}-Y_x^{0,u}\|_{C([0,T_1];\Cper)}^{\tilde p}]\\
&\lesssim_{N,T,\tilde p}
\int_0^{T_1}\E[\|X_x^{\eps,u}-X_x^{0,u}\|_{C([0,t];\Cper)}^{\tilde p}]\dd t.
\end{align*}
Combining all estimates above, we find that for all $0\leq T_1\leq T$,   $\eps_0\in(0,\infty)$, $\eps\in(0,\eps_0]$, $u\in\A_N$, and $x\in B_R^q$, it holds that
\begin{align*}
\E[&\|X_x^{\eps,u}-X_x^{0,u}\|_{C([0,T_1];\Cper)}^{\tilde p}] \\&\lesssim_{\tilde p}\E[\|V_x^{\eps,u}-V_x^{0,u}\|_{C([0,T_1];\Cper)}^{\tilde p}] + \E[\|Y_x^{\eps,u}-Y_x^{0,u}\|_{C([0,T_1];\Cper)}^{\tilde p}] + \eps^{\tilde p/2}\E[\|Z_x^{\eps,u}\|_{C([0,T_1];\Cper)}^{\tilde p}] \\ &\lesssim_{T,\lambda,\tilde p,q,|D|, N,R}\int_0^{T_1} (1+\eps_0^{\tilde p/2})\E[\|X_x^{\eps,u}-X_x^{0,u}\|_{C([0,t];\Cper)}^{\tilde p}]\dd t+\eps^{\tilde p/2}(1+M_*). 
\end{align*} 
Gr\"onwall's inequality gives that 
\[
\E[\|X_x^{\eps,u}-X_x^{0,u}\|_{C([0,T];\Cper)}^{\tilde p}]\leq \eps^{\tilde p/2}C_1,
\]
where $C_1$ is a constant that depends only on $T,\lambda,\tilde p,q,|D|, N,R$, and $\eps_0$. Finally,  Chebychev's inequality provides the convergence in probability as $\eps\to 0$ that allows us to apply Theorem \ref{th:uldp suff}(1), completing the proof of the ULDP. 
\end{proof}

\section{ULDP on $C([0,T];L^p(D))$ over $L^q$-bounded subsets of $L^p(D)$}\label{sec:Lp}

In this section, we prove Theorem \ref{th:Lp}.  We will verify the sufficient criterion from Theorem \ref{th:uldp suff}(2), using an analogous series of estimates as in Section \ref{sec:C}, but with $L^p(D)$ replacing $C(D)$. Eventually, we will prove  uniform bounds for $ \| X_x^{\eps,u} - X_x^{0,u} \|_{L^{\tilde p}(\Om;C([0,T];L^{p}(D)))}$ for a suitable ${\tilde p}>4$. 

Recall the splitting  $X_x^{\eps,u}-X_x^{0,u}=V_x^{\eps,u}-V_x^{0,u}+Y_x^{\eps,u}-Y_x^{0,u}+\sqrt{\eps}Z_x^{\eps,u}$ (see \eqref{eq:split X}). 
Again, we will absorb the differences $V_x^{\eps,u}-V_x^{0,u}$ and $Y_x^{\eps,u}-Y_x^{0,u}$ in our Gr\"onwall estimate, and are left with bounding $Z_x^{\eps,u}$. As in Section \ref{sec:C}, we will reduce this problem to bounding $X_x^{0,u}$ uniformly in $u$, $x$, and $\eps$. The latter will be achieved in Subsection \ref{sub:I2}.

\subsection{Uniform $L^{\lambda\tilde p}\times L^p$-bound for $X_x^{0,u}$}\label{sub:I2} 

The following result will be crucial for controlling the term $Z^{\eps,u}_x$ in the proof of our main result. It can be viewed as an analogue of Lemma \ref{lem:conv Linfty}. 
We note that the stated range for $\tilde p$ is always non-empty: this follows from \eqref{eq:I1 final constraints} when $q<\lambda p$, and holds trivially when $q\geq \lambda p$.

\begin{lemma}\label{lem:I2}
Let Assumption \ref{asm} hold. 
 Let $N>0$. Suppose that $p\in[2,\infty), q\in[1,\infty)$ and $\lambda\in(0,1]$ satisfy 
\begin{equation}\label{eq:I1 final constraints}
\frac{2\lambda p}{p+2}< q.
\end{equation} 
Suppose that $\tilde{p}\in(4,\frac{2pq}{\lambda p-q})$ if $q< \lambda p$, and $\tilde{p}\in(4,\infty)$ if  $q\geq \lambda p$. 

Then, there exists a constant $M\in[0,\infty)$ such that for all $u\in\A_N$ and $x\in L^q(D)$: 
\begin{equation*} 
\int_0^T\|X_x^{0,u}(s)\|_{L^{\lambda p}(D)}^{\lambda\tilde p}\dd s \leq M (1+\|x\|_{L^q(D)}^{\lambda\tilde p})\quad \text{a.s.}
\end{equation*}
The constant $M$ depends only on $T,\lambda,p,\tilde p,q,|D|,$ and $N$.
\end{lemma} 
\begin{proof}
Let $N\in[0,\infty)$ and $x\in L^q(D)$. 
Since $X_x^{0,u}(\om)=X_x^{0,u(\om)}$ for a.e.\ $\om\in\Om$, and by the definition of $\A_N$, it suffices to prove the bound for $u\in L^2(0,T;L^2(D))$ with $\|u\|_{L^2(0,T;L^2(D))}\leq N$. 
We will apply Gr\"onwall's inequality to $T_1\mapsto \|X_x^{0,u}\|_{L^{\lambda\tilde p}(0,T_1;L^{\lambda p}(D))}^{\lambda\tilde p}$. Recall that $X_x^{0,u}=S(\cdot)x+V_x^{0,u}+Y_x^{0,u}$. We estimate these terms separately.  

Thanks to the parameter assumptions, we can apply \eqref{eq:Sx Lp}, giving for all $T_1\in(0,T]$:
\[
\int_0^{T_1}\|S(s)x\|_{L^{\lambda p}(D)}^{\lambda \tilde p}\dd s 
 \lesssim_{T,\lambda,p,\tilde p,q,|D|} \|x\|_{L^q(D)}^{\lambda\tilde p}.
\]  
Moreover, assuming for the moment that $\lambda p\geq 1$ and $\lambda\tilde p\geq 1$, we have by Minkowki's inequality,  H\"older's inequality and the linear growth of $b$:
\begin{align*}
\|V_x^{0,u}(t,\cdot)\|_{L^{\lambda p}(D)}^{\lambda\tilde p}
&\leq \Big(\int_0^t\|S(t-s)b(X_x^{0,u}(s,\cdot))\|_{L^{\lambda p}(D)}\dd s\Big)^{\lambda\tilde p}\\
&\lesssim_{T,\lambda,\tilde p} \int_0^t\|S(t-s)b(X_x^{0,u}(s,\cdot))\|_{L^{\lambda p}(D)}^{\lambda\tilde p}\dd s \\
&\leq \int_0^t\|b(X_x^{0,u}(s,\cdot))\|_{L^{\lambda p}(D)}^{\lambda\tilde p}\dd s\\
&\lesssim_{T,\lambda,p,\tilde p,|D|} 1+ \|X_x^{0,u}\|_{L^{\lambda\tilde p}(0,t;L^{\lambda p}(D))}^{\lambda\tilde p}.
\end{align*}
Integrating over $t$, we conclude that for all $T_1\in(0,T]$:
\begin{align*}
\|V_x^{0,u}\|_{L^{\lambda\tilde p}(0,T_1;L^{\lambda p}(D))}^{\lambda\tilde p}
&\lesssim_{T,\lambda,p,\tilde p,|D|} 1+
\int_0^{T_1}\|X_x^{0,u}\|_{L^{\lambda\tilde p}(0,t;L^{\lambda p}(D))}^{\lambda\tilde p} \dd t.
\end{align*} 
By the last part of Lemma \ref{lem:det conv} applied with $p_1=p_2= p\geq 2$, $\varphi\ceqq \sigma(X_x^{0,u})$, and $\tilde p>4$, the growth bound \eqref{eq:growth} for $\sigma$, and the fact that $\lambda \in (0,1]$, we have for all $t\in[0,T]$:
\begin{align*}
\|Y_x^{0,u}(t,\cdot)\|_{L^{\lambda p}(D)}^{\tilde p} \lesssim_{\lambda,p,\tilde p,|D|}
  \|Y_x^{0,u}(t,\cdot)\|_{L^{p}(D)}^{\tilde p} 
  &\lesssim_{N,T,\tilde p} \int_0^t \|\sigma(X_x^{0,u}(s,\cdot))\|_{L^{p}(D)}^{\tilde p}\dd s\\
  &\lesssim_{\tilde p,|D|} \int_0^t \big(1+\|X_x^{0,u}(s,\cdot)\|_{L^{\lambda p}(D)}^{\lambda \tilde p}\big)\dd s\\
  &\lesssim_{T } 1 +  \|X_x^{0,u}\|_{L^{\lambda\tilde p}(0,t; L^{\lambda p}(D))}^{  \tilde p}.
\end{align*}  
Raising the estimate to the power $\lambda$ and integrating over $t$, we conclude that for all $T_1\in(0,T]$:
\begin{align*}
\|Y_x^{0,u}\|_{L^{\lambda\tilde p}(0,T_1;L^{\lambda p}(D))}^{\lambda\tilde p}
=\int_0^{T_1}\|Y_x^{0,u}(t)\|_{L^{\lambda p}(D)}^{\lambda\tilde p}\dd t
&\lesssim_{\lambda,p,\tilde p,|D|,N,T} 1+
\int_0^{T_1}\|X_x^{0,u}\|_{L^{\lambda\tilde p}(0,t;L^{\lambda p}(D))}^{\lambda\tilde p} \dd t.
\end{align*} 
Combining the estimates above, if $\lambda p\geq 1$ and $\lambda\tilde p\geq 1$, we obtain for all $T_1\in(0,T]$:  
\begin{align*}
  \|&X_x^{0,u}\|_{L^{\lambda\tilde p}(0,{T_1};L^{\lambda p}(D))}^{\lambda\tilde p}\\
  &\lesssim_{\lambda,\tilde p}\|S(\cdot)x\|_{L^{\lambda\tilde p}(0,{T_1};L^{\lambda p}(D))}^{\lambda\tilde p}
  +\|V_x^{0,u}\|_{L^{\lambda\tilde p}(0,{T_1};L^{\lambda p}(D))}^{\lambda\tilde p}+\|Y_x^{0,u}\|_{L^{\lambda\tilde p}(0,{T_1};L^{\lambda p}(D))}^{\lambda\tilde p}\\
  &\lesssim_{T,\lambda,p,\tilde p,q,|D|,N}\|x\|_{L^q(D)}^{\lambda\tilde p}+ 1+
\int_0^{T_1}\|X_x^{0,u}\|_{L^{\lambda\tilde p}(0,t;L^{\lambda p}(D))}^{\lambda\tilde p} \dd t.
\end{align*}
Gr\"onwall's inequality thus gives  
\[
 \|X_x^{0,u}\|_{L^{\lambda\tilde p}(0,T;L^{\lambda p}(D))}^{\lambda\tilde p} 
  \lesssim_{T,\lambda,p,\tilde p,q,|D|,N}\|x\|_{L^q(D)}^{\lambda\tilde p}+ 1,
\]
concluding the proof if $\lambda p\geq 1$ and $\lambda\tilde p\geq 1$. 

If $\lambda p<1$ or $\lambda\tilde p<1$, then we define  $p_*\ceqq p\vee \frac{1}{\lambda}$ and $\tilde p_*\ceqq \tilde p\vee \frac{1}{\lambda}$. We verify that the parameter conditions also hold with $(p_*,\tilde p_*)$. Clearly, $p_*\geq p\geq 2$ and $\tilde p_*\geq \tilde p>4$. Moreover, if $q<\lambda p_*$, then $p_*\neq 1/\lambda$ (since $q\geq 1$), hence $p_*=p$ and $\tilde p_*=\tilde p \vee \frac{1}{\lambda}\leq \tilde p\vee \frac{2}{\lambda}<\frac{2pq}{\lambda p-q}=\frac{2p_*q}{\lambda p_*-q}$. Also, \eqref{eq:cond tilde p} is satisfied for $p_*$: if $p_*=p$, then by assumption, and if $p_*=1/\lambda$, then  $\frac{2\lambda p_*}{p_*+2}=\frac{2\lambda}{1+2\lambda}<1\leq q$. 
Consequently, our result can be applied with $p_*$ and $\tilde p_*$, giving a.s.
\[
\int_0^T \|X_x^{0,u}(s)\|_{L^{\lambda p_*}(D)}^{\lambda \tilde p_*}\dd s\leq M(1+\|x\|_{L^q(D)}^{\lambda\tilde p_*}).
\]
Applying Jensen's inequality twice (for $|\cdot|^{ p/ p_*}$ on $D$ and for $|\cdot|^{\tilde p/\tilde p_*}$ on $[0,T]$), we find 
\begin{align*}
\int_0^T \|X_x^{0,u}(s)\|_{L^{\lambda p}(D)}^{\lambda \tilde p}\dd s
    \lesssim_{T,p,p_*,\tilde p,\tilde p_*,|D|}\Big(\int_0^T \|X_x^{0,u}(s)\|_{L^{\lambda p_*}(D)}^{\lambda \tilde p_*}\dd s\Big)^{\tilde p/\tilde p_*}
    &\leq (M(1+\|x\|_{L^q(D)}^{\lambda\tilde p_*}))^{\tilde p/\tilde p_*}\\
    &\leq M^{\tilde p/\tilde p_*}(1+\|x\|_{L^q(D)}^{\lambda\tilde p}).
\end{align*}
Redefining $M$, and noting that $p_*,\tilde p_*$ depend only on $\lambda,p,$ and $\tilde p$, this proves the result for the general case. 
\end{proof}

\subsection{Conclusion of the ULDP}\label{sub:concl Lp}

By employing a Gr\"onwall scheme and applying Lemma \ref{lem:I2}, we now establish Theorem \ref{th:Lp}. 

\begin{proof}[Proof of Theorem \ref{th:Lp}] 
First, we observe that it suffices to prove the theorem for $q \in(\frac{2\lambda p}{p+2},\lambda p)$. 
Indeed, $\frac{2\lambda p}{p+2}<\lambda p$, so if $  q\geq\lambda p$, then we can pick $\tilde q\in(\frac{2\lambda p}{p+2},\lambda p)$  and note that  $L^{q}(D)$-bounded sets are also $L^{\tilde q}(D)$ bounded (since $D$ is bounded). Thus the  ULDP follows by applying the stated ULDP with $\tilde q$. 

From now on, let $q \in(\frac{2\lambda p}{p+2},\lambda p)$ and $x\in L^q(D)$.  
We will apply Gr\"onwall's inequality to the mapping $T_1\mapsto \E[\| X_x^{\eps,u} - X_x^{0,u} \|_{C([0,T_1];L^{p}(D))}^{{\tilde p}}]$. 
For all $t\in[0,T]$, we have 
\[
\|X_x^{\eps,u}(t)-X_x^{0,u}(t)\|_{L^p(D)}^{{\tilde p}}\lesssim_{{\tilde p}} \|V_x^{\eps,u}(t)-V_x^{0,u}(t)\|_{L^p(D)}^{{\tilde p}}+ \|Y_x^{\eps,u}(t)-Y_x^{0,u}(t)\|_{L^p(D)}^{{\tilde p}}+\eps^{p/2}\|Z_x^{\eps,u}(t)\|_{L^p(D)}^{{\tilde p}}.
\]  
Note that the difference $V_x^{\eps,u}-V_x^{0,u}$ equals $\Psi_x^b(X_x^{\eps,u}-S(\cdot)x)-\Psi_x^b(X_x^{0,u}-S(\cdot)x)$, thus the last part of Lemma \ref{lem:reg_ET} and H\"older's inequality yield for all $T_1\in(0,T]$:  
\begin{align*}
 \E[\|V_x^{\eps,u} - V_x^{0,u} \|_{C([0,T_1];L^{p}(D))}^{\tilde p}]
&\lesssim_{T,\tilde p} 
\Big(\int_0^{T_1} \| X_x^{\eps,u} - X_x^{0,u} \|_{L^{\tilde p}(\Om;C([0,s];L^{p}(D)))}\dd s\Big)^{\tilde p} \\
&\lesssim_{T,\tilde p} \int_0^{T_1} \E\big[ \|X_x^{\eps,u}-X_x^{0,u} \|_{C([0,s];L^{p}(D))}^{{\tilde p}}\big]\dd s.   
\end{align*}  
Similarly, $Y_x^{\eps,u}-Y_x^{0,u}=\Psi_x^{\sigma,u}(X_x^{\eps,u}-S(\cdot)x)-\Psi_x^{\sigma,u}(X_x^{0,u}-S(\cdot)x)$, so the last part of Lemma \ref{lem:reg_ET}   and H\"older's inequality yield 
\begin{align*}
 \E[\| Y_x^{\eps,u} - Y_x^{0,u} \|_{C([0,T_1];L^{p}(D))}^{\tilde p}]
&\lesssim_{N,T,p,\tilde p} 
\int_0^{T_1} \E\big[ \|X_x^{\eps,u}-X_x^{0,u} \|_{C([0,s];L^{p}(D))}^{{\tilde p}}\big]\dd s.   
\end{align*}  
For $Z_x^{\eps,u}$, we can apply Proposition \ref{prop:stoch conv}    with $\Phi\ceqq \sigma(X_x^{\eps,u})$ (so $Z=Z^{\eps,u}_x$). The latter, combined with the growth bound \eqref{eq:growth} for $\sigma$ and the fact that $\lambda\in(0,1]$, gives
\begin{align*}
\E[\|Z_x^{\eps,u}&\|_{C([0,T_1];L^{p}(D))}^{{\tilde p}}]
\lesssim_{T,p,\tilde p,|D|} \E\Big[\int_0^{T_1}(1+\|X_x^{\eps,u}(s)\|_{L^{\lambda p}(D)}^{\lambda\tilde p})\dd s\Big]\\
  &\lesssim_{T,\lambda,p,\tilde p,|D|} \int_0^{T_1}  \E\big[ \|X_x^{\eps,u}-X_x^{0,u} \|_{C([0,s];L^{p}(D))}^{{\tilde p}}\big]\dd s  + \Big(1+\E\Big[\int_0^{T_1}\|X_x^{0,u}(s)\|_{L^{\lambda p}(D)}^{{\lambda\tilde p}}\dd s\Big]\Big). 
  \end{align*} 
Combining the estimates above, we find that for any fixed $\eps_0>0$, for all $\eps\in(0,\eps_0]$, and for all $T_1\in(0,T]$, 
\begin{align*}
  \E[&\| X_x^{\eps,u} - X_x^{0,u} \|_{C([0,T_1];L^{p}(D))}^{{\tilde p}}] \\
  &\lesssim_{{\tilde p}} \E[\| V_x^{\eps,u} - V_x^{0,u} \|_{C([0,T_1];L^{p}(D))}^{{\tilde p}}]+\E[\| Y_x^{\eps,u} - Y_x^{0,u} \|_{C([0,T_1];L^{p}(D))}^{{\tilde p}}]+\eps^{\frac{{\tilde p}}{2}}\E[\|Z_x^{\eps,u}\|_{C([0,T_1];L^{p}(D))}^{{\tilde p}}]\\
  &\lesssim_{N,T,\lambda, p,\tilde p,|D|} \int_0^{T_1} (1+\eps_0^{\frac{\tilde p}{2}})\E\big[ \|X_x^{\eps,u}-X_x^{0,u} \|_{C([0,s];L^{p}(D))}^{{\tilde p}}\big]\dd s  +\eps^{\frac{{\tilde p}}{2}}\Big(1+\E\Big[\int_0^{T_1}\|X_x^{0,u}(s)\|_{L^{\lambda p}(D)}^{{\lambda\tilde p}}\dd s\Big]\Big). 
  \end{align*}
Now, Gr\"onwall's inequality gives 
\begin{equation}\label{eq:gronwallLp}
\E[\| X_x^{\eps,u} - X_x^{0,u} \|_{C([0,T];L^{p}(D))}^{{\tilde p}}]\leq C_1\eps^{\frac{{\tilde p}}{2}}\Big(1+\E\Big[\int_0^T\|X_x^{0,u}(s)\|_{L^{\lambda p}(D)}^{\lambda\tilde p}\dd s\Big]\Big),
\end{equation}
for a constant $C_1$ that depends on $N,T,\lambda, p,\tilde p,|D|,$ and $\eps_0$. 
Moreover, Lemma \ref{lem:I2} gives for all $\tilde{p}\in(4,\frac{2pq}{\lambda p-q})$ (see \eqref{eq:BqR}):  
\begin{equation}\label{eq:toshowYx0eps}
  \sup_{u\in\A_N}\sup_{x\in B_R^q} \E\Big[\int_0^T\|X_x^{0,u}(s)\|_{L^{\lambda p}(D)}^{\lambda\tilde p}\dd s\Big]<\infty.
\end{equation}  
Fixing any $\tilde p$ in the (non-empty) range above, and combining \eqref{eq:gronwallLp} and \eqref{eq:toshowYx0eps} with  Chebychev's inequality proves that the condition of Theorem \ref{th:uldp suff}(2) is satisfied, thus the latter yields the ULDP.  
\end{proof}

\begin{remark}\label{rem:Lq initial data}
In fact, the proofs of Theorem \ref{th:C} and \ref{th:Lp} establish  that the corresponding  condition in Theorem \ref{th:uldp suff} holds uniformly over all $x\in B_R^q$. Nevertheless, the ULDP is stated only for $x\in B_R^q\cap C(D)$ and  $x\in B_R^q\cap L^p(D)$ respectively, because for $x\in L^q(D)\setminus C(D)$ or $x\in L^q(D)\setminus L^p(D)$, the solution $X_x^\eps$ to the SHE may fail to belong to $C([0,T];C(D))$ or $C([0,T];L^p(D))$, respectively. 
\end{remark}

\section{Further ULDPs}\label{sec:additionalULDPs}

In this section, we derive further ULDPs using our previous results and the arguments developed in their proofs. 
Our next  corollary establishes ULDPs for initial data that need only belong to $L^q(D)$, without the additional requirement of belonging to $L^p(D)$, as in Theorem \ref{th:Lp}.  

\begin{corollary}[ULDPs with initial data in $L^q(D)$]\label{cor:uldp pure Lq}  
Let the conditions of Theorem \ref{th:Lp} hold. For $\eps>0$ and $x\in L^q(D)$, let $X_x^{\eps}=X_x^{\eps,0}$ be the mild solution to \eqref{eq:SHE} provided by Theorem \ref{th:wpnLp}. Then:  
\begin{enumerate}
    \item 
    The family $\{X_x^{\eps}-S(\cdot)x:\eps>0,x\in L^q(D)\}$  satisfies the ULDP on $C([0,T];L^{p}(D))$  uniformly over bounded subsets of $L^q(D)$, with rate functions $J_x\col C([0,T];L^{p}(D))\to [0,+\infty]$ given by 
\begin{align*}
J_x(\varphi)\ceqq \frac12\inf\{\|u&\|_{L^2(0,T;L^2(D))}^2 : u\in L^2(0,T;L^2(D)),\, \varphi=\tilde X_x^{0,u} \}, 
\end{align*}
where $\inf\varnothing\ceqq +\infty$ and $\tilde X_x^{0,u}$ is the mild solution to 
\begin{equation*} 
\begin{cases}
\frac{\partial \tilde X}{\partial t}(t, \xi) = \Delta \tilde X(t, \xi) +b\big(\tilde X(t,\xi)+S(t)x(\xi)\big)\\
\qquad\qquad\quad+ \sigma\big(\tilde X(t, \xi)+S(t)x(\xi)\big)u(t,\xi), \quad  \xi \in D, t \in(0,T], \\ 
\tilde X(0, \xi) = 0, \quad\xi\in D.
\end{cases} 
\end{equation*}  
\item If $\mathcal{J}$ is any subset of $L^q(D)$, and $(\mathcal{E},d)$ is any Polish topological vector space such that 
\begin{align}\label{eq:E inclusionJ}
\{S(\cdot)x:x\in \mathcal{J}\}\subset \mathcal{E}, \quad
C([0,T];L^p(D))\subset \mathcal{E},    
\end{align}
and there exists a constant $K$ such that  for all $f,g\in \mathcal{E}$ with $f-g\in C([0,T];L^p(D))$,   
\begin{align}\label{eq:Lip metricJ}
d(f,g)\leq K\|f-g\|_{C([0,T];L^p(D))},  
\end{align} 
then the family  $\{X_x^{\eps} :\eps>0,x\in \mathcal{J}\}$  satisfies the ULDP on $\mathcal{E}$  uniformly over $L^q(D)$-bounded subsets of $\mathcal{J}$, with rate functions $I_x\col \mathcal{E}\to [0,+\infty]$ defined by the formula \eqref{eq:rate}. 
\end{enumerate}
\end{corollary}
\begin{proof}
\textbf{Proof of \textit{(1)}:} 
We apply \cite[Th.\ 2.13]{Salins19equivalences} with the  Polish space $C([0,T];L^p(D))$. 
We have existence of measurable solution maps $\tilde{\G}_x^\eps\col C([0,T];\R^\infty)\to C([0,T];L^p(D))$ such that $\tilde{\G}_x^\eps(W)=X_x^{\eps}-S(\cdot)x$ a.s.,  thanks to Theorem \ref{th:wpnLp} and the Yamada--Watanabe theorem. Indeed, for each fixed $x\in L^q(D)$,  $  X_x^\eps-S(\cdot)x$  is a mild solution to 
\begin{equation}\label{eq:SHE2}
\begin{cases}
\frac{\partial \tilde X}{\partial t}(t, \xi) = \Delta \tilde X(t,\xi) + b\big(S(t)x(\xi)+\tilde X(t,\xi)\big)\\
\qquad\qquad\quad+ \sqrt{\eps}  \sigma\big(S(t)x(\xi)+\tilde X(t, \xi)\big)\dot{W}(t,\xi), \quad  \xi \in D, t \in(0,T], \\ 
\tilde X_x^{\eps}(0, \xi) = 0, \quad\xi\in D.
\end{cases}  
\end{equation} 
Because $\tilde X$ is a mild solution to \eqref{eq:SHE2} if and only if $\tilde X+S(\cdot)x$ is a mild solution to \eqref{eq:SHE}, the pathwise uniqueness stated in Theorem \ref{th:wpnLp} implies that  $X_x^\eps-S(\cdot)x$ is the unique mild solution to \eqref{eq:SHE2}  within $L^{\tilde p}(\Om;C([0,T];L^p(D)))$. The Yamada--Watanabe theorem (see \cite{ondrejat}) yields measurable mappings $\tilde{\G}_x^{\eps}\col C([0,T];\R^\infty)\to C([0,T];L^p(D))$ such that $\tilde{\G}_x^{\eps}(W)=\tilde X_x^{\eps}-S(\cdot)x$ a.s., where we identify $W=(B_k)_{k\in\N}$ as in  \eqref{eq:Walsh int}.  

Taking into account Remark \ref{rem:Lq initial data}, the proof of Theorem \ref{th:Lp}    gives immediately 
\begin{equation}\label{eq:lim2}
\lim_{\eps\downarrow 0}\sup_{x\in B_R^q}\sup_{u\in\A_N}\P(\|X_x^{\eps,u}-X_x^{0,u}\|_{C([0,T];L^p(D))}>\delta)=0, 
\end{equation} Since $X_x^{\eps,u}-S(\cdot)x-(X_x^{0,u}-S(\cdot)x)=X_x^{\eps,u}-X_x^{0,u}$,  
the ULDP now follows by \cite[Th.\ 2.13]{Salins19equivalences}.

\textbf{Proof of \textit{(2)}:} One can combine part \textit{(1)} either with a contraction argument and the definition of the ULDP, or with another application of \cite[Th.\ 3.12]{Salins19equivalences}. Here, we use the latter approach. 

Let $F\ceqq C([0,T];L^p(D))$. The proof of \textit{(1)} yielded    measurable maps 
$\tilde{\G}_x^{\eps}\col C([0,T];\R^\infty)\to F$
such that $\tilde{\G}_x^{\eps}(W)=X_x^{\eps}-S(\cdot)x$ a.s. Now, for any $x\in\mathcal{J}$, we have $S(\cdot)x\in\mathcal{E}$ by assumption, and $\Phi_x\col F\to \mathcal{E}\col \phi\mapsto S(\cdot)x+\phi$ is continuous since $F\into \mathcal{E}$ and since $\mathcal{E}$ is a topological vector space (so translations are continuous). Thus, the mappings ${\G}_x^\eps\ceqq \Phi_x\circ \tilde{\G}_x^\eps\col C([0,T];\R^\infty)\to \mathcal{E}$ are Borel measurable and satisfy $\G_x^{\eps}(W)= X_x^{\eps}$ a.s. 

Furthermore, for $\eps,N\in[0,\infty)$, and $u\in \A_N$, we have  $X_x^{\eps,u}=S(\cdot)x+(X_x^{\eps,u}-S(\cdot)x)\in \mathcal{E}$, using \eqref{eq:E inclusionJ} and the $C([0,T];L^p(D))$-regularity of $X_x^{\eps,u}-S(\cdot)x$ from  Theorem \ref{th:wpnLp}. Moreover, by combining \eqref{eq:lim2} with \eqref{eq:Lip metricJ}, we obtain
\[
\lim_{\eps\downarrow 0}\sup_{x\in B_R^q\cap\mathcal{J}}\sup_{u\in\A_N}\P(d(X_x^{\eps,u},X_x^{0,u})>\delta)=0, \]
thus the stated ULDP thus follows from \cite[Th.\ 3.12]{Salins19equivalences}. 
\end{proof}

Analogous ULDP results can also be established under the conditions of Theorem \ref{th:C}.  

\begin{remark}\label{rem:more uldps}
Corollary \ref{cor:uldp pure Lq} holds analogously   under the conditions of Theorem \ref{th:C}, with every instance of $L^p(D)$ replaced by $C(D)$. The proof is completely analogous, noting that well-posedness was provided by Theorem \ref{th:wpnC} and  taking again note of Remark \ref{rem:Lq initial data}. 
\end{remark}

The following can be derived as a special case of Corollary \ref{cor:uldp pure Lq}\textit{(2)}.

\begin{corollary}[ULDP  in $L^r(0,T;L^p(D))$]\label{cor:uldp LrLp}  
Let the conditions of Theorem \ref{th:Lp} hold. For $\eps>0$ and $x\in L^q(D)$, let $X_x^{\eps}=X_x^{\eps,0}$ be the mild solution to \eqref{eq:SHE} provided by Theorem \ref{th:wpnLp}.  Let $r \in [1,\infty)$ be such that  $\frac{r}{2}(\frac1q-\frac{1}{p})<1$. 

Then, the family  $\{X_x^{\eps} :\eps>0,x\in L^q(D)\}$  satisfies the ULDP on $L^r(0,T;L^{p}(D))$  uniformly over bounded subsets of $L^q(D)$, with rate functions $I_x\col L^r(0,T;L^{p}(D))\to [0,+\infty]$ defined by the formula \eqref{eq:rate}. 
\end{corollary}
\begin{proof}
We apply Corollary \ref{cor:uldp pure Lq}\textit{(2)} with $\mathcal{J}=L^q(D)$. 
    For any $x\in L^q(D)$, we have  $S(\cdot)x\in L^r(0,T;L^{p}(D))$, due to the assumptions on $r$ and the bound \eqref{eq:est}. Furthermore, we have a continuous embedding $C([0,T];L^p(D))\into L^r(0,T;L^{p}(D))$, so both \eqref{eq:E inclusionJ} and \eqref{eq:Lip metricJ} are satisfied. 
\end{proof}

Our final result shows that a ULDP over $L^1(D)$-bounded sets of initial data always holds  in the  $C([0,T];L^r(D))$ topology for $r\leq 2$, whenever the noise has sublinear growth ($\lambda\in(0,1)$). If the noise has linear growth ($\lambda=1$), we obtain the ULDP over $L^q(D)$-bounded sets of initial data for any $q>1$. 

\begin{corollary}[{ULDP on $C([0,T];L^r(D))$ for $r\in[1,2]$}]\label{cor:Lp p<2}  
Let Assumption \ref{asm} hold. Let $r\in[1,2]$ and let $q=1$ if $\lambda<1$, and $q\in(1,r]$ if $\lambda=1$.   

Then, the family $\{X_x^{\eps}:\eps>0,x\in L^r(D)\}$ of mild solutions to \eqref{eq:SHE}  satisfies the ULDP on $C([0,T];L^r(D))$, uniformly over initial data in $L^q(D)$-bounded subsets of $L^r(D)$,  
with the rate functions $I_x$ from Theorem \ref{th:Lp}. 
\end{corollary} 
\begin{proof}[Proof of Corollary \ref{cor:Lp p<2}]
We apply Corollary \ref{cor:uldp pure Lq}\textit{(2)} with $p=2$, $\mathcal{J}\ceqq L^r(D)\subset L^q(D)$  and $\mathcal{E}=C([0,T];L^r(D))$. 
    Because $\lambda<q$, $p=2$ satisfies the parameter condition $\frac{2\lambda  p}{ p+2}=\lambda< q$  of Theorem \ref{th:Lp}. Furthermore, $C([0,T];L^p(D))\into C([0,T];L^r(D))$ since $p=2\geq r$, and it holds that $S(\cdot)x\in L^r(D)$ for all $x\in \mathcal{J}$ since $S$ is a strongly continuous semigroup on $L^r(D)$.  
\end{proof}

\bibliographystyle{plain}
\bibliography{literature}

\begin{thebibliography}{10}

\bibitem{LDP6}
H.~Bessaih and A.~Millet.
\newblock Large deviation principle and inviscid shell models.
\newblock {\em Electron. J. Probab.}, 14, 2009.

\bibitem{LDP7}
H.~Bessaih and A.~Millet.
\newblock Large deviations and the zero viscosity limit for {2D} stochastic {Navier}–{Stokes} equations with free boundary.
\newblock {\em SIAM J. Math. Anal.}, 44(3):1861--1893, 2012.

\bibitem{ULDP1}
A.~Biswas and A.~Budhiraja.
\newblock Exit time and invariant measure asymptotics for small noise constrained diffusions.
\newblock {\em Stoch. Process. Appl.}, 121(5):899--924, 2011.

\bibitem{LDP1}
Z.~Brzezniak, Q.~Li, and T.~Zhang.
\newblock Large deviation principle of stochastic evolution equations with reflection.
\newblock {\em J. Evol. Equ.}, 24(4):91, 2024.

\bibitem{BDM08}
A.~Budhiraja, P.~Dupuis, and V.~Maroulas.
\newblock Large deviations for infinite dimensional stochastic dynamical systems.
\newblock {\em Ann. Probab.}, 36(4), 2008.

\bibitem{ULDP2}
A.~Budhiraja, P.~Dupuis, and V.~Maroulas.
\newblock Large deviations for stochastic flows of diffeomorphisms.
\newblock {\em Bernoulli}, 16(1), 2010.

\bibitem{CCY2024}
D.~Candil, L.~Chen, and C.Y. Lee.
\newblock Parabolic stochastic {PDE}s on bounded domains with rough initial conditions: moment and correlation bounds.
\newblock {\em Stoch. Partial Differ. Equ.: Anal. Comput.}, 12(3):1507--1573, 2024.

\bibitem{LDP8}
C.~Cardon-Weber.
\newblock Large deviations for a {Burgers}’-type {SPDE}.
\newblock {\em Stoch. Process. Appl.}, 84(1):53--70, 1999.

\bibitem{CR04}
S.~Cerrai and M.~Röckner.
\newblock Large deviations for stochastic reaction-diffusion systems with multiplicative noise and non-{Lipshitz} reaction term.
\newblock {\em Ann. Probab.}, 32(1B), 2004.

\bibitem{CD2014}
L.~Chen and R.C. Dalang.
\newblock H{\"o}lder-continuity for the nonlinear stochastic heat equation with rough initial conditions.
\newblock {\em Stoch. Partial Differ. Equ.: Anal. Comput.}, 2(3):316--352, 2014.

\bibitem{CD2015}
L.~Chen and R.C. Dalang.
\newblock Moments and growth indices for the nonlinear stochastic heat equation with rough initial conditions.
\newblock {\em Ann. Probab.}, 43(6):3006 -- 3051, 2015.

\bibitem{CH2019}
L.~Chen and J.~Huang.
\newblock Comparison principle for stochastic heat equation on {$\R^d$}.
\newblock {\em Ann. Probab.}, 47(2):989--1035, 2019.

\bibitem{CH2023}
L.~Chen and J.~Huang.
\newblock Superlinear stochastic heat equation on {$\R^d$}.
\newblock {\em Proc. Am. Math. Soc.}, 151(09):4063--4078, 2023.

\bibitem{CK2019}
L.~Chen and K.~Kim.
\newblock Nonlinear stochastic heat equation driven by spatially colored noise: moments and intermittency.
\newblock {\em Acta Math. Sci.}, 39(3):645--668, 2019.

\bibitem{CX2026}
L.~Chen and P.~Xia.
\newblock Asymptotic properties of stochastic partial differential equations in the sublinear regime.
\newblock {\em Ann. Probab.}, 54(4):1686--1714, 2026.

\bibitem{LDP9}
I.~Chueshov and A.~Millet.
\newblock Stochastic {2D} hydrodynamical type systems: well posedness and large deviations.
\newblock {\em Appl. Math. Optim.}, 61(3):379--420, 2010.

\bibitem{daprato}
G.~Da~Prato and J.~Zabczyk.
\newblock {\em Stochastic equations in infinite dimensions}.
\newblock Number 152 in Encyclopedia of mathematics and its applications. Cambridge university press, Cambridge, 2nd edition, 2014.

\bibitem{DZ10}
A.~Dembo and O.~Zeitouni.
\newblock {\em Large {Deviations} {Techniques} and {Applications}}, volume~38 of {\em Stochastic {Modelling} and {Applied} {Probability}}.
\newblock Springer, Berlin, Heidelberg, 2010.

\bibitem{LDP10}
J.~Duan and A.~Millet.
\newblock Large deviations for the {Boussinesq} equations under random influences.
\newblock {\em Stoch. Process. Appl.}, 119(6):2052--2081, 2009.

\bibitem{FKN2025}
M.~Foondun, D.~Khoshnevisan, and E.~Nualart.
\newblock On the local well-posedness of randomly forced reaction-diffusion equations with {$L^2$} initial data and a superlinear reaction term.
\newblock {\em Probability Theory and Related Fields}, pages 1--35, 2026.

\bibitem{ULDP4}
M.~Foondun and L.~Setayeshgar.
\newblock Large deviations for a class of semilinear stochastic partial differential equations.
\newblock {\em Stat. Probab. Lett.}, 121:143--151, 2017.

\bibitem{freidlin88}
M.I. Freidlin.
\newblock Random perturbations of reaction-diffusion equations: {T}he quasi-deterministic approximation.
\newblock {\em Trans. Am. Math. Soc.}, 305(2):665--697, 1988.

\bibitem{FW3rdedition}
M.I. Freidlin and A.D. Wentzell.
\newblock {\em Random {Perturbations} of {Dynamical} {Systems}}, volume 260 of {\em Grundlehren der mathematischen {Wissenschaften}}.
\newblock Springer Berlin Heidelberg, Berlin, Heidelberg, 2012.

\bibitem{ULDP3}
E.~Gautier.
\newblock Uniform large deviations for the nonlinear {Schrödinger} equation with multiplicative noise.
\newblock {\em Stoch. Process. Appl.}, 115(12):1904--1927, 2005.

\bibitem{LDP4}
W.~Hu, M.~Salins, and K.~Spiliopoulos.
\newblock Large deviations and averaging for systems of slow-fast stochastic reaction–diffusion equations.
\newblock {\em Stoch. PDE: Anal. Comp.}, 7(4):808--874, 2019.

\bibitem{KX96}
G.~Kallianpur and J.~Xiong.
\newblock Large deviations for a class of stochastic partial differential equations.
\newblock {\em Ann. Probab.}, 24(1):320--345, 1996.

\bibitem{LDP3}
A.~Kumar and M.T. Mohan.
\newblock Small time asymptotics for a class of stochastic partial differential equations with fully monotone coefficients forced by multiplicative {Gaussian} noise.
\newblock {\em Stoch. Dyn.}, 25(05):2550021, 2025.

\bibitem{LDP2}
R.~Li, R.~Wang, and B.~Zhang.
\newblock A large deviation principle for the stochastic heat equation with general rough noise.
\newblock {\em J. Theor. Probab.}, 37(1):251--306, 2024.

\bibitem{LDP11}
W.~Liu.
\newblock Large deviations for stochastic evolution equations with small multiplicative noise.
\newblock {\em Appl. Math. Optim.}, 61(1):27--56, 2010.

\bibitem{LDP12}
W.~Liu, M.~Röckner, and X.~Zhu.
\newblock Large deviation principles for the stochastic quasi-geostrophic equations.
\newblock {\em Stoch. Process. Appl.}, 123(8):3299--3327, 2013.

\bibitem{LDP13}
Y.~Lv and A.J. Roberts.
\newblock Large deviation principle for singularly perturbed stochastic damped wave equations.
\newblock {\em Stoch. Anal. Appl.}, 32(1):50--60, 2014.

\bibitem{LDP14}
C.~Mo and J.~Luo.
\newblock Large deviations for stochastic differential delay equations.
\newblock {\em Nonlinear Anal. Theory Methods Appl.}, 80:202--210, 2013.

\bibitem{ondrejat}
M.~Ondrej{\'a}t.
\newblock Uniqueness for stochastic evolution equations in {B}anach spaces.
\newblock {\em Diss. Math.}, 426:1--63, 2004.

\bibitem{LDP15}
V.~Ortiz-López and M.~Sanz-Solé.
\newblock A {L}aplace principle for a stochastic wave equation in spatial dimension three.
\newblock In D.~Crisan, editor, {\em Stochastic {Analysis} 2010}, pages 31--49. Springer Berlin Heidelberg, Berlin, Heidelberg, 2011.

\bibitem{LDP23}
T.~Pan, S.~Shang, J.~Zhai, and T.~Zhang.
\newblock Large deviations of fully local monotone stochastic partial differential equations driven by gradient-dependent noise.
\newblock {\em Bernoulli}, 32(1):249--273, 2026.

\bibitem{LDP16}
J.~Ren, S.~Xu, and X.~Zhang.
\newblock Large deviations for multivalued stochastic differential equations.
\newblock {\em J. Theor. Probab.}, 23(4):1142--1156, 2010.

\bibitem{LDP17}
J.~Ren and X.~Zhang.
\newblock Freidlin–{Wentzell}'s large deviations for stochastic evolution equations.
\newblock {\em J. Funct. Anal.}, 254(12):3148--3172, 2008.

\bibitem{LDP18}
M.~Röckner, T.~Zhang, and X.~Zhang.
\newblock Large deviations for stochastic tamed {3D} {Navier}-{Stokes} equations.
\newblock {\em Appl Math Optim}, 61(2):267--285, 2010.

\bibitem{Salins19equivalences}
M.~Salins.
\newblock Equivalences and counterexamples between several definitions of the uniform large deviations principle.
\newblock {\em Probab. Surveys}, 16:99--142, 2019.

\bibitem{Salins21reacdiff}
M.~Salins.
\newblock Systems of small-noise stochastic reaction-diffusion equations satisfy a large deviations principle that is uniform over all initial data.
\newblock {\em Stoch. Process. Appl.}, 142:159--194, 2021.

\bibitem{Salins25SHE}
M.~Salins.
\newblock Solutions to the stochastic heat equation with polynomially growing multiplicative noise do not explode in the critical regime.
\newblock {\em Ann. Probab.}, 53(1):223--238, 2025.

\bibitem{SWZ2026}
S.~Shang, P.~Wang, and T.~Zhang.
\newblock {$L^2$}-solutions to stochastic reaction-diffusion equations with superlinear drifts driven by space-time white noise.
\newblock {\em Stoch. Process. Appl.}, page 104991, 2026.

\bibitem{Sowers92}
R.B. Sowers.
\newblock Large deviations for a reaction-diffusion equation with non-{G}aussian perturbations.
\newblock {\em Ann. Probab.}, 20(1):504--537, 1992.

\bibitem{LDP19}
S.S. Sritharan and P.~Sundar.
\newblock Large deviations for the two-dimensional {Navier}–{Stokes} equations with multiplicative noise.
\newblock {\em Stoch. Process. Appl.}, 116(11):1636--1659, 2006.

\bibitem{LDP20}
C.~Sun, H.~Gao, J.~Duan, and B.~Schmalfuß.
\newblock Rare events in the {Boussinesq} system with fluctuating dynamical boundary conditions.
\newblock {\em J. Differ. Equ.}, 248(6):1269--1296, 2010.

\bibitem{LDP5}
E.~Theewis and M.~Veraar.
\newblock Large deviations for stochastic evolution equations in the critical variational setting.
\newblock {\em Stoch. Process. Appl.}, 196:104898, 2026.

\bibitem{NVW15}
J.~van Neerven, M.~Veraar, and L.~Weis.
\newblock Stochastic integration in {B}anach spaces -- a survey.
\newblock In {\em Stochastic analysis: a series of lectures}, volume~68, pages 297--332. Springer, Basel, 2015.

\bibitem{LDP21}
T.~Xu and T.~Zhang.
\newblock White noise driven {SPDEs} with reflection: {Existence}, uniqueness and large deviation principles.
\newblock {\em Stoch. Process. Appl.}, 119(10):3453--3470, 2009.

\bibitem{LDP22}
J.~Zhai and T.~Zhang.
\newblock Large deviations for stochastic models of two-dimensional second grade fluids.
\newblock {\em Appl. Math. Optim.}, 75(3):471--498, 2017.

\end{thebibliography}

\end{document}